\documentclass{amsart}

\newcommand{\apref}[3]{\hyperref[#2]{#1\ref*{#2}#3}}

\usepackage{enumerate}
\usepackage[latin1]{inputenc}
\usepackage{dsfont}
\usepackage{amssymb,amsthm,amsmath}

\input{xy}
\xyoption{all}

\usepackage{mathrsfs}

\theoremstyle{plain}
\newtheorem{prop}{Proposition}[section]

\newtheorem{thm}[prop]{Theorem}

\newtheorem*{thmAnn}{Theorem A}

\theoremstyle{definition}

\newtheorem{example}[prop]{Example}

\theoremstyle{remark}
\newtheorem{remark}[prop]{Remark}

\DeclareMathOperator{\PSL}{PSL}

\DeclareMathOperator{\Ima}{Im}
\DeclareMathOperator{\Rea}{Re}

\DeclareMathOperator{\Stab}{Stab}

\DeclareMathOperator{\Ext}{ext}

\DeclareMathOperator{\cyl}{cyl}

\newcommand{\pch}{\mc A}
\newcommand{\fpch}{\mathbb A}

\newcommand{\choices}{\mathbb S}
\newcommand{\shmap}{\mathbb T}

\newcommand{\dec}{\text{dec}}
\newcommand{\parab}{\text{par}}

\newcommand\N{\mathbb{N}}

\newcommand\R{\mathbb{R}}
\newcommand\Z{\mathbb{Z}}
\newcommand\C{\mathbb{C}}

\newcommand{\h}{\mathbb{H}}

\newcommand{\mc}[1]{\mathcal #1}

\newcommand{\wt}{\widetilde}
\newcommand{\wh}{\widehat}

\newcommand{\eps}{\varepsilon}

\DeclareMathOperator{\FE}{FE}

\DeclareMathOperator{\id}{id}

\DeclareMathOperator{\Fct}{Fct}
\DeclareMathOperator{\MCF}{MCF}

\newcommand{\sceq}{\mathrel{\mathop:}=}
\newcommand{\seqc}{\mathrel{=\mkern-4.5mu{\mathop:}}}

\newcommand{\mat}[4]{\begin{pmatrix} #1&#2\\#3&#4\end{pmatrix}}
\newcommand{\bmat}[4]{\begin{bmatrix} #1&#2\\#3&#4\end{bmatrix}}

\newcommand{\textbmat}[4]{\left[\begin{smallmatrix} #1&#2 \\ #3&#4
\end{smallmatrix}\right]}

\usepackage[colorlinks]{hyperref}
\usepackage{graphics}
\usepackage[T1]{fontenc} 

\begin{document}

\title{A dynamical approach to Maass cusp forms}
\author[A.\@ Pohl]{Anke D.\@ Pohl}
\address{Mathematisches Institut, Georg-August-Universit\"at G\"ottingen,  Bunsenstr. 3-5, 37073 G\"ottingen}
\email{pohl@uni-math.gwdg.de}
\begin{abstract} 
For nonuniform cofinite Fuchsian groups $\Gamma$ which satisfy a certain additional geometric condition, we show that the Maass cusp forms for $\Gamma$ are isomorphic to $1$-eigenfunctions of a finite-term transfer operator. The isomorphism is constructive.
\end{abstract}
\thanks{The author was supported by the ERC Starting Grant ANTHOS}
\maketitle


\section{Introduction}

Let $\Gamma$ be a nonuniform cofinite Fuchsian group and consider its action on the hyperbolic plane $\h$ by M\"obius transformations. The purpose of this article is to characterize, under a certain additional geometric requirement on $\Gamma$, the Maass cusp forms for $\Gamma$ as $1$-eigenfunctions of a finite-term transfer operator which arises from a discretization of the geodesic flow on $\Gamma\backslash\h$.

Maass cusp forms for $\Gamma$ are specific eigenfunctions of the Laplace-Beltrami operator $\Delta$ acting on $L^2(\Gamma\backslash\h)$ which decay rapidly towards any cusp of $\Gamma\backslash\h$. They span the cuspidal spectrum of $\Delta$ in $L^2(\Gamma\backslash\h)$, which together with the residual spectrum spans the discrete spectrum. These forms play important roles in various fields, e.g.\@ harmonic analysis, number theory, mathematical physics and quantum chaos. For lattices with special symmetries (e.g.\@ congruence subgroups), infinitely many linearly independent Maass cusp forms are known to exist, even though only few could be provided by explicit formulas, yet. For generic nonuniform lattices, according to the Phillips-Sarnak conjecture \cite{Phillips_Sarnak_cuspforms, Phillips_Sarnak_weyl}, existence of Maass cusp forms is expected to be very limited.

The transfer operator approach has its origin in thermodynamic formalism and was initially primarily used to investigate dynamical zeta functions \cite{Ruelle_dynzeta, Ruelle}. An important example of such a dynamical zeta function is the Selberg zeta function
\[
 Z(s) \sceq \prod_\gamma \prod_{k=0}^\infty \left( 1- e^{-(s+k)\ell(\gamma)}\right),
\]
converging on some right halfplane. Here, the outer product runs over all primitive periodic geodesics $\gamma$ on $\Gamma\backslash\h$, and $\ell(\gamma)$ denotes the length of $\gamma$. As known from classical Selberg theory, the Selberg zeta function encodes the eigenvalues of Maass cusp forms in some of the zeros of its meromorphic continuation. In terms of quantum chaos, the Selberg zeta function allows to correlate some properties of some quantum objects (eigenvalues of Maass cusp forms) with some properties of purely classical objects (geodesic length spectrum). 

For the modular surface $\PSL_2(\Z)\backslash\h$, a deeper relation has been established by the following transfer operator approach. Mayer \cite{Mayer_zeta, Mayer_thermo, Mayer_thermoPSL} showed that the transfer operator $\mc L_s$ with parameter $s\in\C$ of the Gauss map,
\[
 \mc L_sf(z) = \sum_{n\in\N} \frac{1}{(z+n)^{2s}}f\left(\frac{1}{z+n}\right),
\]
is a nuclear operator of order zero on a certain Banach space of holomorphic functions, its Fredholm determinant represents the Selberg zeta function
\begin{equation}\label{fredrep}
 Z(s) = \det(1-\mc L_s^2) = \det(1-\mc L_s)\det(1+\mc L_s)
\end{equation}
on some right halfplane $\Rea s > s_0$, and the family $(\mc L_s)_{\Rea s>s_0}$ admits a meromorphic continuation to all $s\in\C$. The Gauss map and hence the transfer operator arise from a specific discretization of the geodesic flow on the modular surface \cite{Artin, Series}. The $\pm 1$-eigenfunctions of $\mc L_s$ ($1>\Rea s>0$) are isomorphic to highly regular solutions of the functional equation \cite{Chang_Mayer_transop, Lewis_Zagier}
\begin{equation}\label{classfeq}
 f(x) = f(x+1) + (x+1)^{-2s} f\left(\frac{x}{x+1}\right),\quad x\in \R_{>0}.
\end{equation}
By Lewis-Zagier \cite{Lewis_Zagier} (see \cite{Bruggeman} for an alternative approach, and \cite{Lewis} for an earlier result on even Maass cusp forms), these solutions are isomorphic to the Maass cusp forms for $\PSL_2(\Z)$ with eigenvalue $s(1-s)$. The factorization in \eqref{fredrep} relates to even resp.\@ odd Maass cusp forms (see \cite{Efrat_spectral} for a spectral result prior to \cite{Lewis_Zagier}). Hence, in this approach, Maass cusp forms themselves, not only their eigenvalues, can be recovered from the geodesic flow. In addition, the meromorphic continuation of the Selberg zeta function and the relation between its zeros and the eigenvalues of Maass cusp forms can be rededuced. However, the functional equation of $Z$ could not be found via transfer operators, yet. The solutions to \eqref{classfeq} of sufficient regularity were coined \textit{period functions} in \cite{Lewis_Zagier}, in analogy to period polynomials in the Eichler-Manin-Shimura theory.

In the sequel, also other automorphic forms for $\PSL_2(\Z)$ could be included into this picture \cite{Chang_Mayer_period, Chang_Mayer_transop}.

Up to date, the complete picture from a representation of the Selberg zeta function as a Fredholm determinant of a transfer operator family via period functions to Maass cusp forms could only be extended to some finite index subgroups of $\PSL_2(\Z)$ \cite{Deitmar_Hilgert, Chang_Mayer_eigen, Chang_Mayer_extension, Hilgert_Mayer_Movasati, Fraczek_Mayer_Muehlenbruch}. Moreover, there is a second transfer operator approach to $\PSL_2(\Z)$ \cite{Mayer_Muehlenbruch_Stroemberg, Bruggeman_Muehlenbruch} using the discretization in \cite{Mayer_Stroemberg}. From the point of view of quantum chaos, the most significant component of such a transfer operator approach is to establish that certain eigenfunctions of the transfer operator are isomorphic to Maass cusp forms. The discretization in \cite{Pohl_Symdyn2d} for the geodesic flow on $\Gamma\backslash\h$ for various subgroups $\Gamma$ of $\PSL_2(\R)$ has been developed specifically for such a ``short-cut'' transfer operator approach to Maass cusp forms. Common 
to the approaches mentioned above is that the transfer operators involve infinitely many terms. In contrast, the transfer operators arising from \cite{Pohl_Symdyn2d} are finite-term. More precisely, the transfer operator with parameter $s\in\C$ is given by a finite sum of specific elements of $\Gamma$ acting via the action of principal series representation with spectral parameter $s$ on functions which are defined on certain intervals in the geodesic boundary of $\h$. It is shown in \cite{Moeller_Pohl} for Hecke triangle groups and in \cite{Pohl_mcf_Gamma0p} for $\Gamma_0(p)$, $p$ prime, that the $1$-eigenfunctions of these transfer operators for these lattices are period functions and are isomorphic to Maass cusp forms via an integral transform, taking advantage of the characterization of Maass cusp forms in parabolic $1$-cohomology in \cite{BLZ_part2}. The functional equations replacing \eqref{classfeq} are just the defining equations for $1$-eigenfunctions of these transfer operators. Moreover, for Hecke 
triangle groups, \cite{Moeller_Pohl} shows a representation of the Selberg zeta function as in \eqref{fredrep}, including the factorization, for certain modified transfer operators. In \cite{Pohl_spectral}, a spectral result analogous to \cite{Efrat_spectral} is provided.

Here we generalize the dynamical approach to Maass cusp forms to nonuniform lattices $\Gamma$ in $\PSL_2(\R)$, which satisfy a certain additional geometric requirement. Our main result is as follows:

\begin{thmAnn}
Let $s\in\C$, $0< \Rea s < 1$. Then the space of Maass cusp forms for $\Gamma$ with eigenvalue $s(1-s)$ is isomorphic to the space of sufficiently regular $1$-eigenfunctions of the transfer operator with parameter $s$.
\end{thmAnn}

The regularity required for the eigenfunctions is specified in Theorem~\ref{finaliso} below. For the proof of Theorem~A we use the characterization of Maass cusp forms in parabolic $1$-cohomology in \cite{BLZ_part2} and show that the para\-bolic $1$-cocycle classes are isomorphic to these highly regular $1$-eigenfunctions of the transfer operator. Both of these isomorphisms are constructive, and hence the isomorphism in Theorem A is so. 

The discretization of the geodesic flow allows for a number of choices, each choice giving rise to a definition of period functions. By Theorem~A, all these spaces of period functions are isomorphic. The precise effect of the choices in the discretization is discussed in Section~\ref{sec_choices} below, where we also provide an explicit formula for the isomorphism between the different spaces of period functions.

In Section~\ref{sec_prelims} below we recall the discretization of the geodesic flow, present the associated transfer operators and the definition of period functions, and provide the necessary background on the parabolic $1$-cohomology characterization of Maass cusp forms. Theorem~A is then proved in Section~\ref{sec_isom}.

We provide an ongoing example to illustrate all constructions. Other examples can be found in \cite{Pohl_Symdyn2d, Hilgert_Pohl, Moeller_Pohl, Pohl_mcf_Gamma0p}.

A representation as in \eqref{fredrep} of the Selberg zeta function as a Fredholm determinant of a transfer operator family (not necessarily with a factorization) could be established for various lattices $\Gamma$ in $\PSL_2(\R)$ \cite{Pollicott, Fried_triangle, Morita_transfer, Chang_Mayer_extension, Mayer_Muehlenbruch_Stroemberg, Moeller_Pohl}, using several discretization for the geodesic flow on $\Gamma\backslash\h$. In some approaches, the representation is not exact. It rather requires a ``correction factor'' which accounts for double-coding in the discretization \cite{Pollicott, Morita_transfer, Mayer_Muehlenbruch_Stroemberg}. We expect that the transfer operator families used in the article at hand allow for a similar modification as in \cite{Moeller_Pohl}, but we leave this for future investigations. 

Concerning period functions, in addition to the already mentioned results from \cite{Deitmar_Hilgert, Chang_Mayer_extension, BLZ_part2, Moeller_Pohl, Pohl_mcf_Gamma0p}, the $1$-eigenfunctions of the (infinite-term) transfer operators for Hecke triangle groups in \cite{Mayer_Muehlenbruch_Stroemberg} are shown there to satisfy finite-term functional equations. However, for Hecke triangle groups other than $\PSL_2(\Z)$, their relation to Maass cusp forms is still an open problem. In Section~\ref{sec_choices} we compare our period functions to these results. 

\textit{Acknowledgment.} The author thanks the referee for helpful comments.

\section{Symbolic dynamics, transfer operators, and period functions}\label{sec_prelims}

This section serves to present the additional geometric condition required of the considered Fuchsian lattices $\Gamma$ and to briefly recall the discretization of the geodesic flow on $\Gamma\backslash\h$ from \cite{Pohl_Symdyn2d} as well as the characterization of Maass cusp forms in parabolic $1$-cohomology from \cite{BLZ_part2}. For proofs we refer to the original articles. Moreover, we provide a definition of period functions.

To simplify the exposition, we use the upper half plane
\[
 \h \sceq \{ z\in \C \mid \Ima z > 0\}
\]
as model for the hyperbolic plane and identify its geodesic boundary with $P^1(\R) \cong \R \cup \{\infty\}$. In this model, the group of orientation-preserving Riemannian isometries on $\h$ can be identified with $\PSL_2(\R)$, whose action on $\h$ is given by fractional linear transformations and extends continuously to $P^1(\R)$. Thus we have
\[
 \bmat{a}{b}{c}{d}.z = 
\begin{cases}
 \frac{az+b}{cz+d} & \text{if $cz+d \not=0$}
\\
\infty & \text{if $cz+d=0$}
\end{cases}
\quad
\text{and}
\quad
\bmat{a}{b}{c}{d}.\infty =
\begin{cases}
\frac{a}{c} & \text{if $c\not=0$}
\\
\infty & \text{if $c=0$}
\end{cases}
\]
for $\textbmat{a}{b}{c}{d}\in\PSL(2,\R)$ and $z\in \h \cup \R$. Throughout let $\Gamma$ be a nonuniform cofinite Fuchsian group and suppose that $\infty$ is a representative of a cusp of $\Gamma\backslash \h$. Then the stabilizer group $\Gamma_\infty = \Stab_\Gamma(\infty)$ of $\infty$ in $\Gamma$ is generated by some element
\[
 T \sceq \bmat{1}{\lambda}{0}{1} \quad\in\Gamma
\]
with $\lambda > 0$. A point in $P^1(\R)$ is called \textit{cuspidal} if it is fixed by a parabolic element in $\Gamma$. We use $S\h$ to denote the unit tangent bundle of $\h$. The action of $\Gamma$ extends to $S\h$. By $\Gamma\backslash\h$ resp.\@ $\Gamma\backslash S\h$ we denote the quotient space of the $\Gamma$-action on $\h$ resp.\@ $S\h$. We remark that we may identify the unit tangent bundle of $\Gamma\backslash\h$ with $\Gamma\backslash S\h$. If $U$ is a subset of $\h$, then we let $\partial U$ denote its boundary. The complement of a set $B$ in a set $A$ is denoted by 
\[
 A\smallsetminus B = \{ a\in A \mid a\notin B\}.
\]
Finally, a smooth function always refers to a $C^\infty$ function.

\subsection{The additional requirement on $\Gamma$}

The additional condition we require to be satisfied by $\Gamma$ is of geometric nature and restricts the admissible boundary structure of the subset of $\h$ which is common to all exteriors of isometric spheres of $\Gamma$. In short, it says that there is a Ford fundamental domain for $\Gamma$ constructed with respect to $\infty$ such that the highest points of all non-vertical bounding complete geodesic segments are contained in the boundary of the fundamental domain but are not intersection points of two non-vertical sides of the fundamental domain.

Let $g=\textbmat{a}{b}{c}{d}\in\Gamma\smallsetminus\Gamma_\infty$. Then the \textit{isometric sphere} of $g$ is the set
\[
 I(g) \sceq \{ z\in\h \mid |cz+d|=1 \}.
\]
It is identical to the complete geodesic segment connecting $\frac{-d-1}{c}$ and $\frac{-d+1}{c}$, or the semi-circle in $\h$ with center $-\frac{d}{c}$ and radius $\frac{1}{|c|}$. The \textit{exterior} of $I(g)$ is
\[
 \Ext I(g) \sceq \{ z\in\h \mid |cz+d|>1\}.
\]
The \textit{summit} of $I(g)$ is the point
\[
 s = -\frac{d}{c} + \frac{i}{|c|} \quad \in\h.
\]
Let 
\[
 \mc K \sceq \bigcap_{g\in\Gamma\smallsetminus\Gamma_\infty} \Ext I(g)
\]
be the common part of all exteriors of the isometric spheres of $\Gamma$. This is a convex subset of $\h$ which contains
\[
 \{ z\in\h \mid \Ima z > y_0\}
\]
for a sufficiently large $y_0>0$ and whose boundary is a locally finite union of geodesic segments which are connected subsets of isometric spheres. An isometric sphere of $\Gamma$ is called \textit{relevant} if it coincides with the boundary of $\mc K$ in more than one point. From now on we impose the following condition on $\Gamma$:
\begin{equation}\label{A}\tag{A}
\begin{minipage}{10cm}
If for $g\in\Gamma\smallsetminus\Gamma_\infty$ the isometric sphere $I(g)$ is relevant, then its summit is contained in $\partial\mc K$ but is not a vertex of $\mc K$.
\end{minipage}
\end{equation}
For $r\in\R$, let $\mc F_\infty(r) \sceq (r, r+\lambda)+i\R_{>0}$. Then 
\[
 \mc F(r) \sceq \mc F_\infty(r) \cap \mc K
\]
is a Ford fundamental domain for $\Gamma$. If we choose for $r$ the center of a relevant isometric sphere, then $\mc F(r)$ is a fundamental domain as described above.

\begin{example}
Condition~(A) is satisfied by several Fuchsian lattices. One such lattice is provided in Example~\ref{ex_Gamma}, which we will use as an ongoing example for all constructions. Other lattices satisfying (A) are  e.g.\@ Hecke triangle groups and various congruence subgroups. In the following we present a lattice which does not satisfy (A). 

Let $\Lambda$ be the lattice in $\PSL_2(\R)$ which is generated by the three elements
\[
 \mat{1}{\frac{17}{11}}{0}{1},\quad \mat{18}{-5}{11}{-3} \quad\text{and}\quad \mat{3}{-1}{10}{3}.
\]
A Ford fundamental domain for $\Lambda$ is displayed in Figure~\ref{NA_funddom}. We refer to \cite{Pohl_Symdyn2d} for proofs. It clearly shows that (A) does not hold for $\Lambda$. 
 \begin{figure}[h]
\begin{center}
\includegraphics{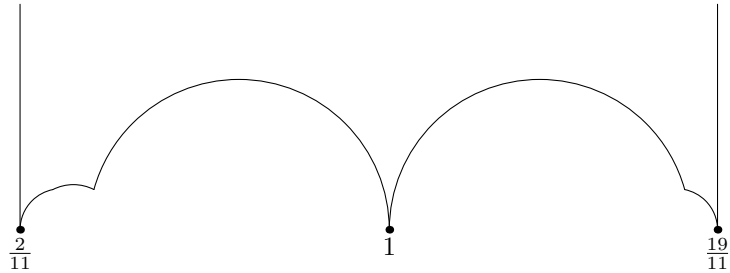}
\end{center}
\caption{Ford fundamental domain for $\Lambda$}
\label{NA_funddom}
\end{figure}
\end{example}

\subsection{Discretization}\label{sec_discr}

The starting point in \cite{Pohl_Symdyn2d} for the discretization of the geodesic flow on $\Gamma\backslash\h$ is a specific choice of a cross section in the sense of Poincar\'e for this flow. This cross section is a subset $\wh C$ on the unit tangent bundle $\Gamma\backslash S\h$ of $\Gamma\backslash \h$ which is intersected by almost all geodesics infinitely often in the past and the future, and each such intersection is discrete in time. Geodesics are here parametrized by arc length, and ``almost all'' refers to all geodesics which do not converge to a cusp forward or backward in time. Before we expound the construction of $\wh C$ in Section~\ref{sec_cross} below, we briefly explain how it gives rise to a discrete dynamical system on subsets of $\R$.

For $\wh v\in \wh C$ let $\wh \gamma_v$ denote the geodesic on $\Gamma\backslash\h$ determined by 
\[
 \frac{d}{dt}\vert_{t=0} \wh\gamma_v(t) = \wh v.
\]
The choice of $\wh C$ yields that if $\wh\gamma_v$ does not converge to a cusp, there is a minimal return time of $\wh\gamma_v$ to $\wh C$, that is  a minimal time $t(\wh v) > 0$ such that 
\[
 \frac{d}{dt}\vert_{t=t(\wh v)} \wh\gamma_v(t) \in \wh C.
\]
Therefore $\wh C$ induces the (partially defined) first return map
\[
 \mc R \colon \wh C \to \wh C,\quad \wh v \mapsto \frac{d}{dt}\vert_{t=t(\wh v)} \wh\gamma_v(t).
\]
The precise domain of definition for $\mc R$ is discussed in detail in \cite{Pohl_Symdyn2d}. The sole property of this domain used hiddenly here is its density in $\wh C$. 

A major property of the cross section $\wh C$ is that it has a set of representatives $C'$ in $S\h$ which decomposes (uniquely) into a finite number $C'_1,\ldots, C'_k$ of subsets each of which is either of the form
\[
 C'_j = \left\{ X\in S\h \left\vert\ X = a\frac{\partial}{\partial x}\vert_{r_j + iy} + b \frac{\partial}{\partial y}\vert_{r_j+iy},\ a>0,\ b\in\R,\ y>0\right.\right\}
\]
for some cuspidal point $r_j\in\R$, or
\[
 C'_j = \left\{ X\in S\h \left\vert\ X = a\frac{\partial}{\partial x}\vert_{r_j+iy} + b\frac{\partial}{\partial y}\vert_{r_j+iy},\ a<0,\ b\in\R,\ y>0\right.\right\}
\]
for some cuspidal point $r_j\in\R$. In other words, each $C'_j$ consists of the unit tangent vectors in $S\h$ which are based on the complete geodesic segment $(r_j,\infty)$ and point into either of the halfspaces $\{ \Rea z > r_j\}$ or $\{\Rea z < r_j\}$. For a given subset $U$ of $\h$ and a unit tangent vector $v\in S\h$, we say that $v$ \textit{points into} $U$ if the geodesic $\gamma_v$ determined by 
\[
 \frac{d}{dt}\vert_{t=0} \gamma_v(t) = v
\]
immediately runs into $U$. More precisely, if there exists $\eta>0$ such that $\gamma_v( (0,\eta) ) \subseteq U$.

The combination of this partition of $C'$ and the first return map $\mc R$ allows to define a discrete dynamical system on parts of the geodesic boundary of $\h$ which is conjugate to $\mc R$. To that end, let
\[
 \wt D_j \sceq \{ (\gamma_v(-\infty), \gamma_v(+\infty), j) \mid v\in C'_j\} \quad \text{for $j=1,\ldots, k$,}
\]
and
\[
 \wt D \sceq \coprod_{j=1}^k \wt D_j.
\]
Then the map $\tau \colon \wh C \to\wt D$ defined by 
\[
 \tau(\wh v) \sceq (\gamma_v(-\infty), \gamma_v(+\infty), j) \quad\text{if $\wh v$ is represented by $v\in C'_j$}
\]
is a bijection. Thus, there is a unique (partially defined) self-map $\wt F$ on $\wt D$ which is conjugate to $\mc R$ via $\tau$. Its domain of definition corresponds to the domain of definition of $\mc R$. 

For each $j\in\{1,\ldots,k\}$, the structure of $C'_j$ yields that $\wt D_j$ is either of the form
\[
 (-\infty, r_j) \times (r_j,\infty) \times \{j\} \quad\text{or}\quad (r_j,\infty)\times (-\infty,r_j)\times \{j\}.
\]
As shown in \cite{Pohl_Symdyn2d}, the map $\wt F$ is locally given by M\"obius transformations of specific elements from $\Gamma$.

From the discrete dynamical system $(\wt D, \wt F)$ we will only need its expanding direction, which means its projection to the last two components. We denote this restricted discrete dynamical system with $(D,F)$, where
\[
 D\sceq \coprod_{j=1}^k D_j,\quad D_j \sceq \{ (\gamma_v(\infty), j) \mid v\in C'_j\}
\]
and $F$ is the self-map of $D$ which is induced by $\wt F$. To be precise, $F$ is a self-map only on
\[
 D\smallsetminus \{ (r,j) \mid \text{$r$ cuspidal, $j=1,\ldots, k$}\}
\]
and can be analytically extended to a map defined on $D$ up to finitely many points. Here we work with this analytic extension and still write $D$ for its domain of definition. It will always be clear on which points of $D$ the map $F$ is not defined.

\subsection{Cross section and choice of set of representatives}\label{sec_cross}

For the definition of the cross section $\wh C$ and a choice of its set $C'$ of representatives we consider $\mc K$ as a subset of $\h \cup P^1(\R)$. 

The vertices of $\mc K$ which are contained in $\h$ are called \textit{inner vertices}, those which are contained in $\R$ are called \textit{infinite vertices}. We decompose its closure $\overline{\mc K}$ as follows into a collection of hyperbolic triangles and rectangles. If $v\not=\infty$ is a vertex of $\mc K$, then $v$ is either the intersection point or the common endpoint of two (uniquely determined) relevant isometric spheres. Let $s_1$ resp.\@ $s_2$ be their summits.  If $v$ is an inner vertex, then we form the hyperbolic rectangle with vertices $\infty$, $s_1$, $v$ and $s_2$. 
If $v$ is an infinite vertex, then we form the two hyperbolic triangles with vertices $\infty$, $v$ and $s_1$ resp.\@  $\infty$, $v$ and $s_2$. In any case, if a side of these triangles and rectangles has $\infty$ as one endpoint, then we call it \textit{vertical}, otherwise \textit{non-vertical}. Let $\mc C$ denote this collection of hyperbolic triangles and rectangles. 

Let $\wt C$ be the set of unit tangent vectors $X \in S\h$  such that $X$ is based on a vertical side of an element in $\mc C$ but is not tangent to this side. Further let $\pi\colon S\h\to \Gamma\backslash S\h$ denote the quotient map. Then we choose
\[
 \wh C \sceq \pi(\wt C) 
\]
as cross section for the geodesic flow on $\Gamma\backslash S\h$. To find a set of representatives for $\wh C$ with the properties announced in Section~\ref{sec_discr} we proceed as follows. 

The elements of $\mc C$ (now considered as subsets of $\h$) provide a tesselation of $\h$. This means their $\Gamma$-translates cover $\h$ and, whenever two $\Gamma$-translates of elements in $\mc C$ have a point in common, then it is either a single point which is a common vertex of both translates or they coincide at a common side or they are equal. Out of the family $\mc C$ we pick a subfamily $\fpch$ of triangles and rectangles whose union forms a (closed) fundamental region for $\Gamma$ in $\h$. Within the family $\fpch$, the tesselation property induces a unique and well-defined side-pairing. In analogy with Poincar\'e's Fundamental Polyhedron Theorem we use this side-pairing to define cycles as explained in the following. For this we remark that non-vertical sides of rectangles (resp.\@ triangles) in $\fpch$ can only be paired with non-vertical sides of rectangles (resp.\@ triangles).

Let $\pch\in\fpch$ be a rectangle. Suppose that $v$ is the inner vertex of $\mc K$ to which $\pch$ is associated, and let $b_1,b_2$ denote the two non-vertical sides of $\pch$. We denote by $k_1(\pch), k_2(\pch)$ the two elements in $\Gamma\smallsetminus\Gamma_\infty$ such that $b_j\in I(k_j(\pch))$ and $k_j(\pch)b_j$ is a non-vertical side of some rectangle in $\fpch$. These are the side-pairing elements for the non-vertical sides of $\pch$. We define $\pch(v)\sceq \pch$. To any rectangle $\pch=\pch(v) \in \fpch$ and any choice $h\in \{k_1(\pch), k_2(\pch)\}$ we assign a finite sequence in $\fpch\times\Gamma$ using the following algorithm:

Set $v_1\sceq v$, $\pch_1\sceq \pch(v_1)$, $h_1\sceq h$, $g_1\sceq\id$ and $g_2 \sceq h_1$. Iteratively for $j=2,3,\ldots$ set $v_j\sceq g_j(v)$ and $\pch_j\sceq \pch(v_j)$. Let $h_j$ be the element in $\Gamma\smallsetminus\Gamma_\infty$ such that $\{h_j,h_{j-1}^{-1}\} = \{k_1(\pch_j),k_2(\pch_j)\}$. Set $g_{j+1} \sceq h_jg_j$. The algorithm stops if $g_{j+1} = \id$. We assign to $(\pch_1,h_1)$ the sequence (the \textit{cycle}) $\big( (\pch_j,h_j) \big)_{j=1,\ldots,k}$ where $k>2$ is minimal such that $g_{k+1} = \id$. 

We consider the two sequences determined by $(\pch, k_1(\pch))$ and $(\pch, k_2(\pch))$ as equivalent, as well as any sequences determined by any element $(\pch', h')$ of these sequences. 

Let $\pch\in\fpch$ be a triangle and let $b$ be its non-vertical side. Then there are unique elements $h\in\Gamma\smallsetminus\Gamma_\infty$ and $\pch'\in\fpch$ such that $b\in I(h)$ and $hb$ is the non-vertical side of $\pch'$. We assign to $(\pch,h)$ the \textit{cycle} $\big( (\pch, h), (\pch',h^{-1})\big)$, which we consider to be equivalent to the cycle $\big( (\pch',h^{-1}), (\pch,h)\big)$.

For any of these cycles in $\fpch\times\Gamma$ we call any element $(\pch,h)$ contained in it a \textit{generator} of the sequence or its equivalence class. To define a set of representatives $C'$ for $\wh C$ we fix a generator for each equivalence class of cycles. Let $\choices$ denote the set of chosen generators.

Let $(\pch,h)$ be an element of a cycle in $\fpch\times\Gamma$. Then one of the vertical sides of $\pch$ is contained in the geodesic segment $(h^{-1}.\infty,\infty)$. We define
\[
 \eps(\pch,h)\sceq
\begin{cases}
 +1 & \text{if $\pch \subseteq \{ z\in\h \mid \Rea z \geq h^{-1}.\infty\}$,}
\\
-1 & \text{if $\pch\subseteq \{ z\in\h \mid \Rea z \leq h^{-1}.\infty\}$.}
\end{cases}
\]
Let $(\pch,h)\in\choices$. Suppose first that $\pch$ is a rectangle. Let $\big( (\pch_j,h_j)\big)_{j=1,\ldots,k}$ be the cycle in $\fpch\times\Gamma$ determined by $(\pch,h)$. Let
\[
 \cyl(\pch)\sceq \min\big(\{\ell\in \{1,\ldots, k-1\} \mid \pch_{\ell+1} = \pch\} \cup \{k\}\big).
\]
For $j=1,\ldots,\cyl(\pch)$ we set
\[
 C'_{(\pch,h),j} \sceq 
\left\{ a\frac{\partial}{\partial x}\vert_{h_j^{-1}.\infty + iy} + b\frac{\partial}{\partial y}\vert_{h_j^{-1}.\infty +iy} \in S\h\ \left\vert\ \eps_j\cdot a>0,\ b\in\R,\ y>0 \vphantom{\frac{\partial}{\partial x}\vert_{h_j^{-1}.\infty + iy}} \right.\right\} 
\]
where $\eps_j \sceq \eps(\pch_j,h_j)$.
Suppose now that $\pch$ is a triangle and let $\big( (\pch,h), (\pch',h^{-1})\big)$ be the cycle determined by $(\pch,h)$. Let $v, v'\in\R$ be the infinite vertices of $\mc K$ to which $\pch$ resp.\@ $\pch'$ are associated. Choose an integer $m=m(\pch,h)\in\Z$ and set $\eps\sceq \eps(\pch,h)$. We define
\begin{align*}
C'_{(\pch,h),1} & \sceq \left\{ a\frac{\partial}{\partial x}\vert_{v + iy} + b\frac{\partial}{\partial y}\vert_{v +iy} \in S\h\ \left\vert\ \eps\cdot a<0,\ b\in\R,\ y>0 \vphantom{\frac{\partial}{\partial x}\vert_{h_j^{-1}.\infty + iy}} \right.\right\}, 
\\
C'_{(\pch,h),2} & \sceq \left\{ a\frac{\partial}{\partial x}\vert_{v' + iy} + b\frac{\partial}{\partial y}\vert_{v' +iy} \in S\h\ \left\vert\ \eps\cdot a>0,\ b\in\R,\ y>0 \vphantom{\frac{\partial}{\partial x}\vert_{h_j^{-1}.\infty + iy}} \right.\right\}, 
\\
C'_{(\pch,h),3} & \sceq \left\{ a\frac{\partial}{\partial x}\vert_{T^mh^{-1}.\infty + iy} + b\frac{\partial}{\partial y}\vert_{T^mh^{-1}.\infty +iy} \in S\h\ \left\vert\ \eps\cdot a>0,\ b\in\R,\ y>0 \vphantom{\frac{\partial}{\partial x}\vert_{h_j^{-1}.\infty + iy}} \right.\right\}. 
\end{align*}
The choice of the integer $m$ will affect the subsequent steps. We record it with the map $\shmap\colon (\pch,h) \mapsto m(\pch,h)$ defined on the elements $(\pch,h)\in\choices$ for which $\pch$ is a triangle. Then
\[
 C'\sceq \bigcup_{\stackrel{(\pch,h)\in\choices}{\text{$\pch$ rectangle}}} \bigcup_{j=1}^{\cyl(\pch)} C'_{(\pch,h),j} \cup \bigcup_{\stackrel{(\pch,h)\in\choices}{\text{$\pch$ triangle}}} \bigcup_{j=1}^{3} C'_{(\pch,h),j}
\]
is a set of representatives for $\wh C$ with the properties described in Section~\ref{sec_discr}.

\begin{example}\label{ex_Gamma}
Let 
\[
 g\sceq \mat{1+\sqrt{2}}{-\frac13(2+\sqrt{2})}{3}{-1}\quad\text{and}\quad T\sceq \mat{1}{\frac13(2+\sqrt{2})}{0}{1}.
\]
We consider the lattice $\Gamma$ in $\PSL_2(\R)$, which is generated by $g$ and $T$. A presentation of $\Gamma$ is e.g.\@ $\Gamma=\langle g,T \mid g^4 = \id\rangle$. This lattice satisfies condition (A). Figure~\ref{ex_funddom} shows a Ford fundamental domain for $\Gamma$ together with its tiling into hyperbolic triangles and rectangles as described above. 
\begin{figure}[h]
\begin{center}
\includegraphics{example.1}
\end{center}
\caption{Fundamental domain for $\Gamma$ with $\fpch = \{\pch_1,\pch_2,\pch_3\}$}
\label{ex_funddom}
\end{figure}
The side $s_1$ is mapped to $s_2$ by $T$, and the side $s_3$ is mapped to $s_4$ by $g$. This lattice has two cusps, represented by $\infty$ and $0$, and a single elliptic point (represented by the only inner vertex), which is of order $4$. We choose $\fpch = \{\pch_1,\pch_2,\pch_3\}$. Then representatives of the equivalence classes of cycles in $\fpch\times\Gamma$ are given by
\[
 \left( (\pch_1,g), (\pch_3,g^{-1}) \right) \quad\text{and}\quad \left( (\pch_2,g),(\pch_2,g), (\pch_2,g), (\pch_2,g)\right).
\]
We remark that $\cyl(\pch_2) = 1$. We choose $\choices = \{ (\pch_1,g), (\pch_2,g) \}$ and $\shmap\colon(\pch_1,g)\mapsto 0$. Figure~\ref{ex_forward} indicates the induced set of representatives for the cross section.
\end{example}

\subsection{The induced discrete dynamical system $(D,F)$}

Let 
\begin{align*}
 \Sigma &\sceq \left\{ \big((\pch,h),j\big) \left\vert\  \text{$(\pch,h)\in\choices$, $\pch$ rectangle, $j=1,\ldots,\cyl(\pch)$}\vphantom{\big((\pch,h),j\big)}\right.\right\} 
\\
& \qquad \cup \left\{ \big((\pch,h),j\big) \left\vert\ \text{$(\pch,h)\in\choices$, $\pch$ triangle, $j=1,2,3$}\vphantom{\big((\pch,h),j\big)}\right.\right\}
\end{align*}
denote the arising \textit{set of symbols}. Then
\[
 C' = \coprod_{\alpha\in\Sigma} C'_\alpha.
\]
We call the sets $C'_\alpha$, $\alpha\in\Sigma$, the \textit{components} of $C'$.

To simplify notations, we use the following conventions. For $r\in\R$ and $\eps \in \{\pm 1\}$ we let
\[
\langle r,\eps\infty\rangle \sceq
\begin{cases}
(r,\infty) & \text{if $\eps = +1$}
\\
(-\infty,r) & \text{if $\eps=-1$.}
\end{cases}
\]

Let $\alpha\in\Sigma$. Recall that $C'_\alpha$ is based on the complete geodesic segment $(r_\alpha,\infty)$ for a cuspidal point $r_\alpha\in\R$. We define
\[
 \eps_\alpha \sceq 
\begin{cases}
+1 & \text{if the elements of $C'_\alpha$ point into $\{\Rea z > r_\alpha\}$},
\\
-1 & \text{if the elements of $C'_\alpha$ point into $\{\Rea z < r_\alpha\}$}.
\end{cases}
\]
Further, we let
\[
 I_\alpha \sceq \langle r_\alpha, \eps_\alpha\infty\rangle \quad\text{and}\quad D_\alpha \sceq I_\alpha\times \{\alpha\}.
\]
Then we have
\[
 D = \coprod_{\alpha\in\Sigma} D_\alpha.
\]
An explicit expression for the map $F\colon D\to D$ can be deduced as follows. Given a point $(r,\alpha) \in D_\alpha$ for some $\alpha\in\Sigma$ we pick any element $v\in C'_\alpha$ such that $\gamma_v(\infty) = r$. Let $\wh\gamma_v = \Gamma.\gamma_v$ be the corresponding geodesic on $\Gamma\backslash\h$ and let $t_0$ be the first return time of $\wh\gamma_v$ to $\wh C$. Then $\gamma'_v(t_0)$ is contained in a (unique) $\Gamma$-translate of some component of $C'$, say $\gamma'_v(t_0) \in g. C'_\beta$. Thus, 
\[
 F(r,\alpha ) = (g^{-1}.r,\beta).
\]
The exact values for $g$ and $\beta$ can be algorithmically calculated from the side-pairing in $\fpch$. The outcome is that $F$ restricts to a finite number of local diffeomorphisms of the form
\[
 \big(I_\alpha \cap g.I_\beta\big) \times \{\alpha\} \to I_\beta \times \{\beta\},\quad (r,\alpha) \mapsto (g^{-1}.r,\beta).
\]
For details we refer to \cite{Pohl_Symdyn2d}.

\begin{example}\label{ex_Gamma_forward}
For the lattice $\Gamma$ in Example~\ref{ex_Gamma} with the choices of $\fpch$, $\choices$ and $\shmap$ done there, we have
\[
 \Sigma = \big\{ \left( (\pch_1,g), 1\right), \left( (\pch_1,g),2\right), \left( (\pch_1,g),3\right), \left( (\pch_2,g),1\right) \big\}.
\]
Moreover, we have
\begin{align*}
D_{(\pch_1,g),1} &= (0,\infty) \times \big\{ \left( (\pch_1,g),1\right)\big\},
\\
D_{(\pch_1,g),2} &= \left(-\infty, \tfrac13(2+\sqrt{2})\right) \times \big\{\left( (\pch_1,g),2\right)\big\},
\\
D_{(\pch_1,g),3} &= \left(-\infty,\tfrac13\right) \times \big\{ \left( (\pch_1,g),3\right)\big\},
\\
D_{(\pch_2,g),1} &= \left(\tfrac13,\infty\right) \times \big\{ \left( (\pch_2,g),1\right) \big\}.
\end{align*}
The local diffeomorphisms which define the induced discrete dynamical system $(D,F)$ can easily be read off from Figure~\ref{ex_forward}, where
\[
 k_1 \sceq g^{-1}T = \mat{1}{0}{3}{1} \quad\text{and}\quad k_2 \sceq gT^{-1} = \mat{1+\sqrt{2}}{-2-\frac43\sqrt{2}}{3}{-3-\sqrt{2}}.
\]
For example, we have
\[
 \left(0,\tfrac13\right)\times \left\{ \left( (\pch_1,g), 1\right)\right\} \to D_{(\pch_1,g),1},\quad \big( r, ( (\pch_1,g),1 )\big) \mapsto \big( k_1^{-1}.r, ( (\pch_1,g),1 )\big)
\]
and
\[
 \left(\tfrac13,\infty\right)\times \left\{ \left( (\pch_1,g),1 \right) \right\} \to D_{(\pch_2,g),1},\quad \big( r, ( (\pch_1,g),1 )\big) \mapsto \big( r, ( (\pch_2,g),1 )\big).
\]
\begin{figure}[h]
\begin{center}
\includegraphics{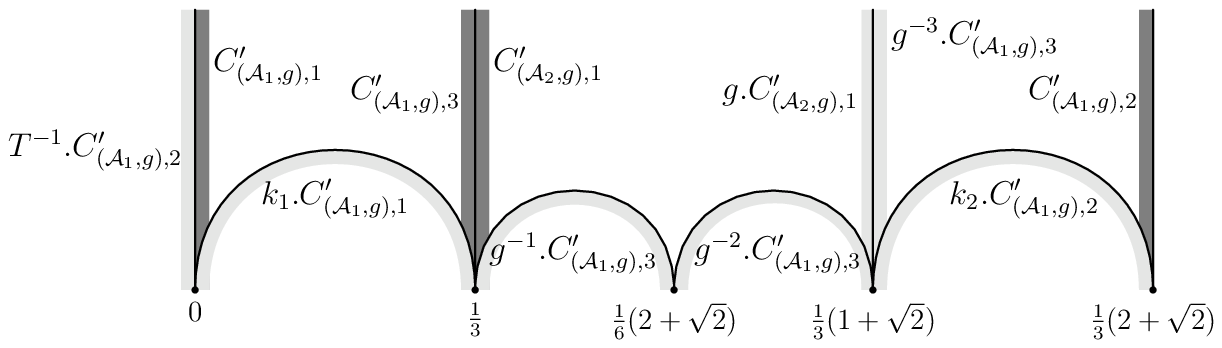}
\end{center}
\caption{Set of representatives for cross section (gray) and forward intersections (light gray)}
\label{ex_forward}
\end{figure}
\end{example}

\subsection{The associated family of transfer operators}
For each $s\in\C$, the transfer operator with parameter $s$ associated to the discrete dynamical system $(D,F)$ is the operator 
\[
 \left(\mc L_{F,s} f\right)(x) \sceq \sum_{y\in F^{-1}(x)} \frac{f(y)}{|F'(y)|^s}
\]
defined on the space $\Fct(D;\C)$ of complex-valued functions on $D$. In this section we provide a matrix representation for $\mc L_{F,s}$.

For any function $\varphi\colon V\to \R$ on some subset $V$ of $\R$ and for any $g\in\PSL_2(\R)$ we set
\begin{equation}\label{tau}
 \big(\tau_s(g^{-1})\varphi\big)(r) \sceq \big( g'(r) \big)^s \varphi(g.r)
\end{equation}
whenever it is well-defined. For appropriate sets $V$, the map $\tau_s$ is a left action of $\PSL_2(\R)$ or $\Gamma$. It is essentially (depending on conventions) the left-action variant of the so-called slash action. 

For $\alpha\in\Sigma$ and $f\in\Fct(D;\C)$ we let 
\[
 f_\alpha\sceq f\cdot 1_{D_\alpha},
\]
where $1_{D_\alpha}$ denotes the characteristic function of the set $D_\alpha$. Then $f= \sum_{\alpha\in\Sigma} f_\alpha$. We may identify $f$ with the vector $(f_\alpha)_{\alpha\in\Sigma}$ and then $D_\alpha$ with $I_\alpha$. Let
\[
 (\wt f_\alpha) = \wt f \sceq \mc L_{F,s}f.
\]
We derive explicit expressions for $\wt f_\alpha$, $\alpha\in\Sigma$, in dependence of the component functions $f_\beta$, $\beta\in\Sigma$.

Let $\alpha\in\Sigma$. Let $v\in C'_\alpha$ and suppose that 
\[
 \gamma'_v( (-\infty,0) )\cap \Gamma.C'\not=\emptyset.
\]
Then there exists a \textit{previous time of intersection}, namely
\[
 t_1 \sceq \max\{ t< 0 \mid \gamma'_v(t) \in \Gamma.C'\}.
\]
We call $\gamma'_v(t_1)$ the \textit{previous intersection} of $v$. To determine for a given $x=(r,\alpha)\in D_\alpha$ its preimages under $F$ is equivalent to determine for the vectors in $C'_\alpha$ their locations of previous intersection. These can be deduced from the side-pairing in $\fpch$ as explained in the following. For $C'_\alpha$, there exists a unique generator $(\pch, h) \in \choices$ and a unique element $(\pch', h')$ in the cycle determined by $(\pch,h)$ and a unique $n\in\Z$ such that one of the vertical sides of $T^n\pch'$ is contained in the base set of $C'_\alpha$ and 
\[
 T^n\pch' \subseteq \{ z\in\h \mid \eps_\alpha\cdot (\Rea z-  r_\alpha) < 0\}.
\]
In other words, some subset of $C'_\alpha$ is based on a vertical side of $T^n\pch'$ and the elements of $C'_\alpha$ do not point into $T^n\pch'$. In this case, we say that $C'_\alpha$ is \textit{neighboring} $T^n\pch'$. We have to distinguish the following three situations: 

\textbf{Situation (1):} $\pch'$ is a rectangle. Then we are in one of the situations shown in Figure~\ref{sit1}.

\begin{figure}[h]
\begin{center}
\includegraphics{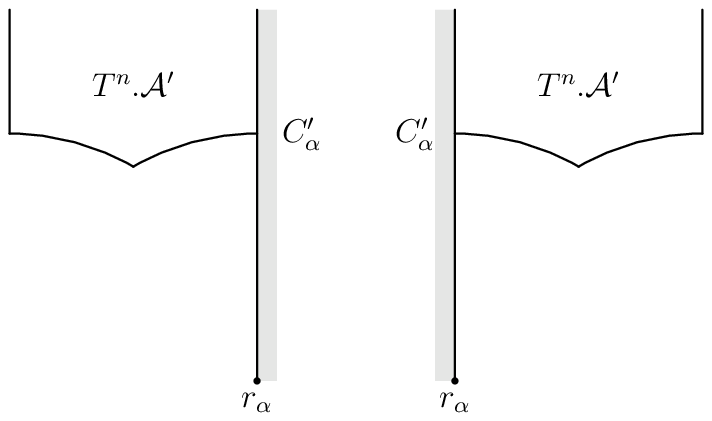}
\end{center}
\caption{Situation (1)}
\label{sit1}
\end{figure}

\textbf{Situation (2):} $\pch'$ is a triangle and $C'_\alpha$ is neighboring $T^n\pch'$ on its long side. Then we are in one of the situations shown in Figure~\ref{sit2}.

\begin{figure}[h]
\begin{center}
\includegraphics{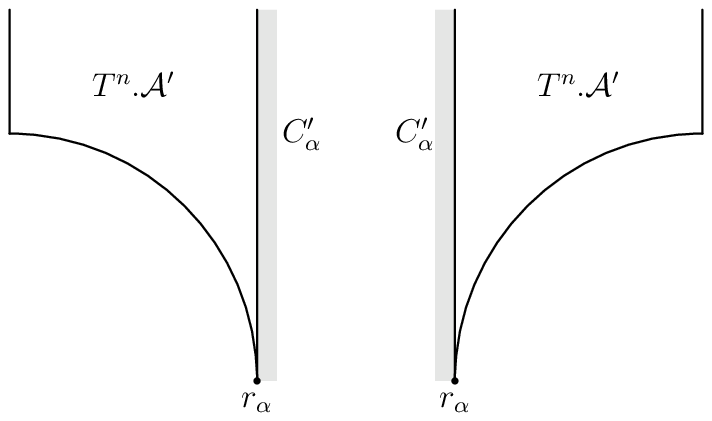}
\end{center}
\caption{Situation (2)}
\label{sit2}
\end{figure}

\textbf{Situation (3):} $\pch'$ is a triangle and $C'_\alpha$ is neighboring $T^n\pch'$ on its short side. Then we are in one of the situations shown in Figure~\ref{sit3}.

\begin{figure}[h]
\begin{center}
\includegraphics{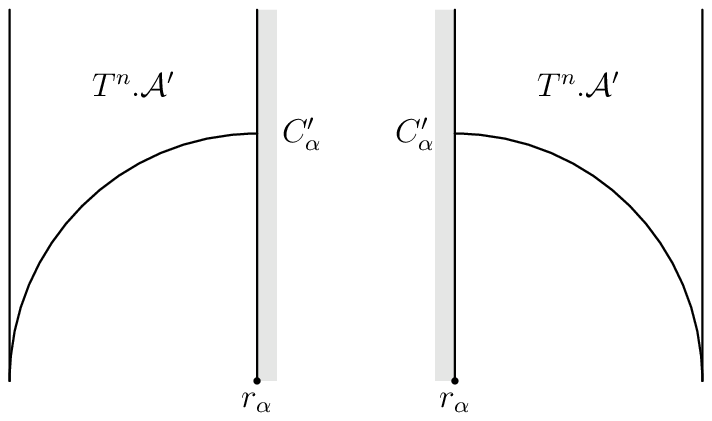}
\end{center}
\caption{Situation (3)}
\label{sit3}
\end{figure}

Note that the two sub-situations shown in any of the Figures~\ref{sit1}-\ref{sit3} are mirror-inverted and thus are equivalent for all further considerations. We will not distinguish these in the following figures. However, their differences are taken into account in all formulas by $\eps_\alpha$. For the locations of the previous intersections we have to subdivide these situations. All further numbering of situations will refer to and extend the ones just introduced.

\textbf{Situation (1):} Let $\big( (\pch_\ell,h_\ell) \big)_{\ell=1,\ldots,k}$ be the cycle determined by $(\pch,h)$ and suppose that $\pch'=\pch_j$ for some $j\in\{1,\ldots, k\}$. We set 
\[
C'_\ell\sceq C'_{(\pch,h),\ell},\quad f_\ell\sceq f_{(\pch,h),\ell}\quad\text{and}\quad \eps_\ell\sceq \eps_{(\pch,h),\ell} 
\]
for $\ell = 1,\ldots, k$, where the index $\ell$ is understood modulo $\cyl(\pch)$. According to the rotation direction of the cycle related to the orientation of $C'_\alpha$ we are either in Situation (1a) shown in Figure~\ref{sit1a} ($\eps_\alpha = \eps_j$) or in Situation (1b) shown in Figure~\ref{sit1b} ($\eps_\alpha = -\eps_j$).

\begin{figure}[h]
\begin{center}
\includegraphics{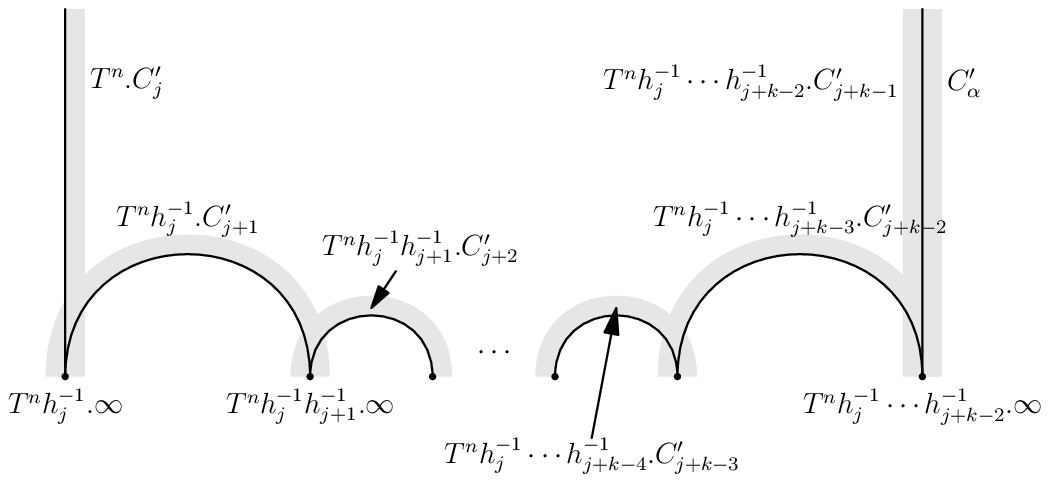}
\end{center}
\caption{Situation (1a)}
\label{sit1a}
\end{figure}

\begin{figure}[h]
\begin{center}
\includegraphics{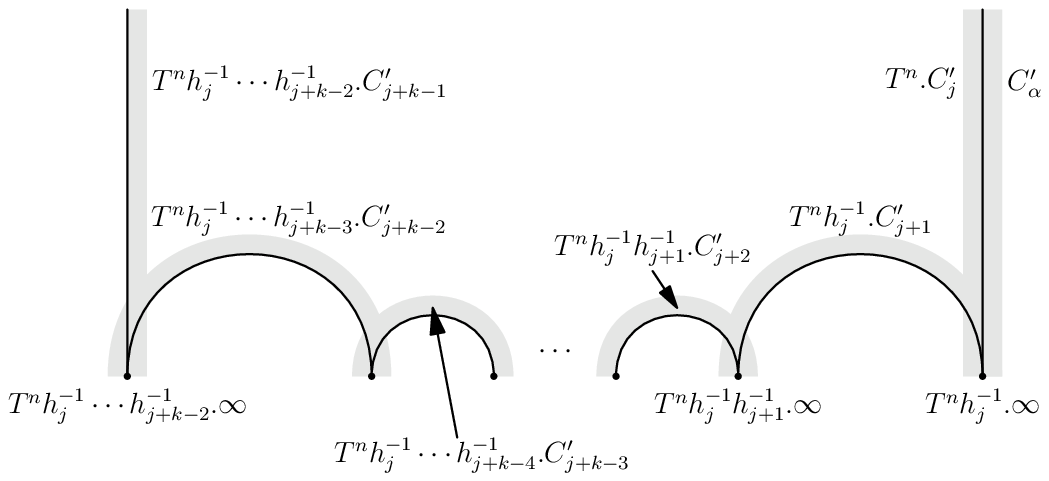}
\end{center}
\caption{Situation (1b)}
\label{sit1b}
\end{figure}

In Situation (1a) we have
\begin{align*}
\wt f_\alpha & = \tau_s(T^n)f_j + \tau_s(T^nh_j^{-1})f_{j+1} + \cdots + \tau_s(T^nh_j^{-1}\ldots h_{j+k-3}^{-1})f_{j+k-2}
\\
& = \sum_{\ell=0}^{k-2} \tau_s(T^nh_j^{-1}\cdots h_{j+\ell-1}^{-1}) f_{j+\ell},
\end{align*}
whereas in Situation (1b) we have
\begin{align*}
\wt f_\alpha & = \tau_s(T^nh_j^{-1})f_{j+1} + \tau_s(T^nh_j^{-1}h_{j+1}^{-1})f_{j+2} + \cdots + \tau_s(T^nh_j^{-1}\ldots h_{j+k-2}^{-1})f_{j+k-1}
\\
& = \sum_{\ell=0}^{k-2} \tau_s(T^nh_j^{-1}\cdots h_{j+\ell}^{-1})f_{j+\ell+1}.
\end{align*}

\textbf{Situation (2):} Let $\big( (\pch,h), (\pch', h^{-1}) \big)$ be the cycle determined by $(\pch, h)$ (note that this $\pch'$ is not necessarily the $\pch'$ from above) and let $m\sceq m(\pch, h)$. We set $C'_j\sceq C'_{(\pch,h),j}$, $f_j\sceq f_{(\pch,h),j}$ and $\eps_j\sceq \eps_{(\pch,h),j}$ for $j=1,2,3$. Then we have either Situation (2a) ($\eps_\alpha = \eps_1$) or Situation (2b) ($\eps_\alpha=-\eps_1$) shown in Figure~\ref{sit2ab}. 

\begin{figure}[h]
\begin{center}
\includegraphics{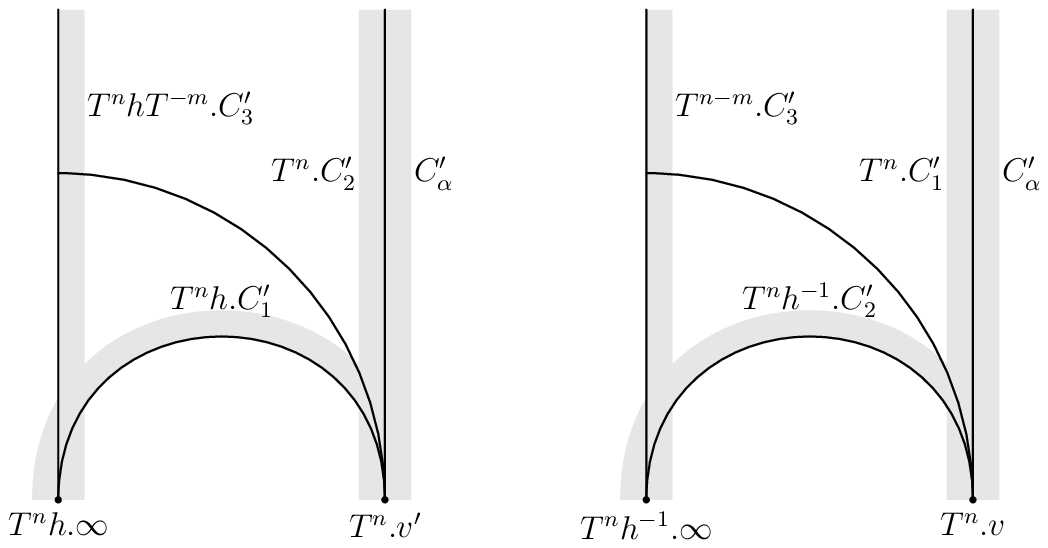}
\end{center}
\caption{On the left Situation (2a),  on the right Situation (2b)}
\label{sit2ab}
\end{figure}

In Situation (2a) we have
\[
 \wt f_\alpha = \tau_s(T^nh)f_1 + \tau_s(T^nhT^{-m})f_3,
\]
and in Situation (2b) we have
\[
 \wt f_\alpha = \tau_s(T^nh^{-1})f_2 + \tau_s(T^{n-m})f_3.
\]

\textbf{Situation (3):} We use the notation from Situation (2). Then we  have either Situation (3a) ($\eps_\alpha = \eps_1$)
or Situation (3b) ($\eps_\alpha=-\eps_1$) shown in Figure~\ref{sit3ab}. 

\begin{figure}[h]
\begin{center}
\includegraphics{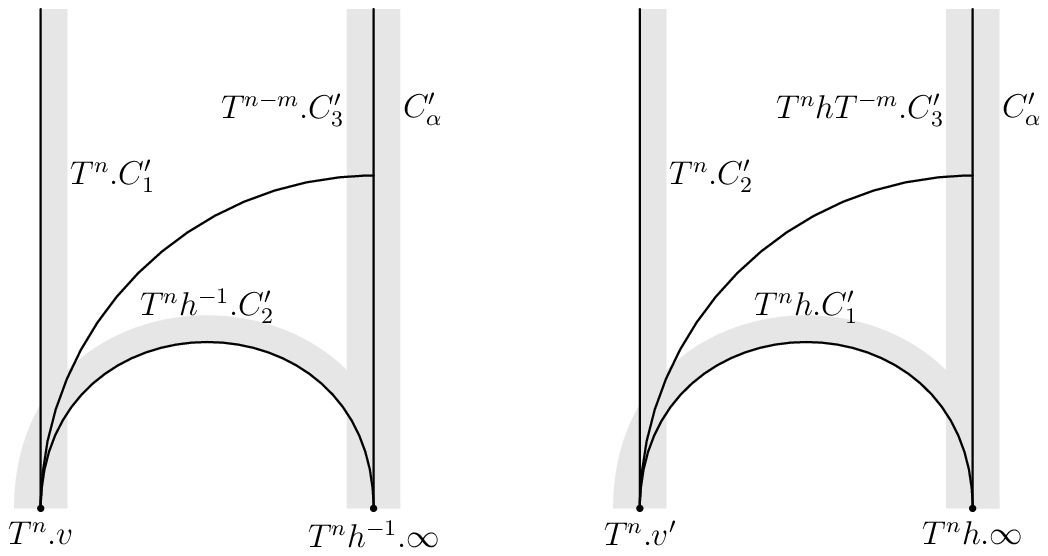}
\end{center}
\caption{On the left Situation (3a),  on the right Situation (3b)}
\label{sit3ab}
\end{figure}

In Situation (3a) we have
\[
\wt f_\alpha = \tau_s(T^n)f_1 + \tau_s(T^nh^{-1})f_2,
\]
whereas in Situation (3b) we have
\[
\wt f_\alpha = \tau_s(T^nh)f_1 + \tau_s(T^n)f_2.
\]

\begin{example}\label{ex_Gamma_backward}
Continuing Example~\ref{ex_Gamma_forward}, we deduce the associated family of transfer operators. The locations of previous intersections can easily be read off from Figure~\ref{ex_backward}.
\begin{figure}[h]
\begin{center}
\includegraphics{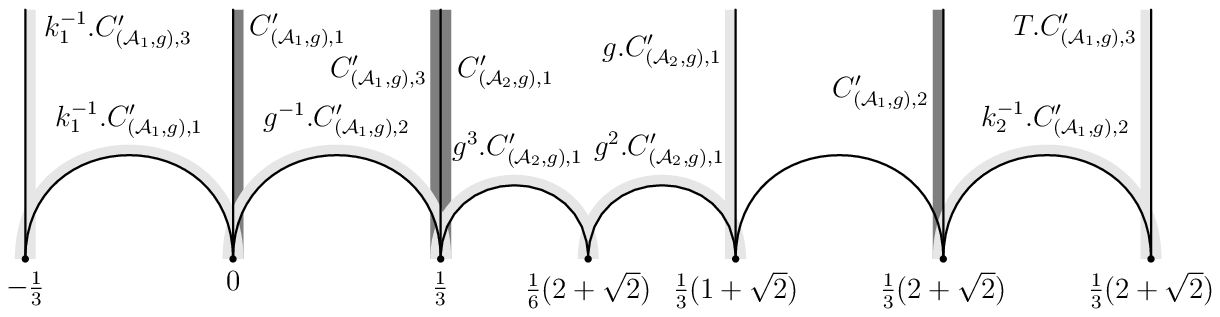}
\end{center}
\caption{Set of representatives for cross section (gray) and backward intersections (light gray)}
\label{ex_backward}
\end{figure}
Using the identification
\[
 f = \begin{pmatrix} f_{(\pch_1,g),1} \\ f_{(\pch_1,g),2} \\ f_{(\pch_1,g),3} \\ f_{(\pch_2,g),1}\end{pmatrix}
\]
(in this order), the associated transfer operator $\mc L_s$ with parameter $s$ is represented by the matrix
\[
 \mc L_s = 
\begin{pmatrix}
\tau_s(k_1^{-1}) & 0 & \tau_s(k_1^{-1}) & 0
\\
0 & \tau_s(k_2^{-1}) & \tau_s(T) & 0
\\
0 & 0 & 0 & \tau_s(g) + \tau_s(g^2) + \tau_s(g^3)
\\
1 & \tau_s(g^{-1}) & 0 & 0  
\end{pmatrix}.
\]
\end{example}

\subsection{Period functions}\label{sec_per}
Let $s\in\C$. We say that a function $\varphi\colon \R\to\C$ extends smoothly ($C^\infty$) to $P^1(\R)$ if for some (and indeed any) element $g\in\PSL_2(\R)$, $g.\infty\not=\infty$, the functions $\varphi$ and $\tau_s(g)\varphi$ are smooth on $\R$. Note that this notion of smooth extension depends on $s$.

For any $\alpha\in\Sigma$, we find a unique $\beta\in\Sigma$ and unique $g\in\Gamma$ such that $C'_\alpha$ and $g.C'_\beta$ are based on the same geodesic segment but are disjoint. We call $(\beta,g)$ the \textit{tuple assigned to} $\alpha$. The precise values for $\beta$ and $g$ can be read off from Situations (1a)--(3b). For $f\in \Fct(D;\C)$ we define
\begin{equation}\label{psialpha}
 \psi_{\alpha,f}\sceq
\begin{cases}
\eps_\alpha f_\alpha & \text{on $\langle r_\alpha, \eps_\alpha\infty\rangle$}
\\
-\eps_\alpha\tau_s(g) f_\beta & \text{on $\langle r_\alpha, -\eps_\alpha\infty\rangle$.}
\end{cases}
\end{equation}

The space of \textit{period functions} $\FE_s^{\omega,\dec}(\Gamma;\fpch,\choices,\shmap)$ (relative to the choices of $\fpch$, $\choices$ and $\shmap$) is defined to be the space of function vectors
\[
 f = (f_\alpha)_{\alpha\in\Sigma}
\]
such that 
\begin{itemize}
\item[(PF1)] $f_\alpha \in C^\omega(I_\alpha;\C)$ for $\alpha\in \Sigma$,
\item[(PF2)] $f=\mc L_{F,s}f$.
\item[(PF3)] If for $\alpha\in\Sigma$, the map $\psi_{\alpha,f}$ in \eqref{psialpha} arises from Situation (1a) or (3b) (that is, $\beta$ and $g$ are determined by these situations), then $\psi_{\alpha,f}$ extends smoothly to $\R$.
\item[(PF4)] If for $\alpha\in\Sigma$, the map $\psi_{\alpha,f}$ is not determined by Situation (1a) or (3b), then it extends smoothly to $P^1(\R)$.
\end{itemize}

\begin{remark}\label{issmooth}
If (PF3) is satisfied, then the maps considered there actually extend smoothly to $P^1(\R)$ by the following argument: If 
\[
\psi_\alpha=
\begin{cases}
\eps_\alpha f_\alpha & \text{on $\langle r_\alpha,\eps_\alpha\infty\rangle$}
\\
-\eps_\alpha \tau_s(g) f_\beta & \text{on $\langle r_\alpha, -\eps_\alpha\infty\rangle$}
\end{cases}
\]
is one of these maps, then 
\[
\psi_\beta=
\begin{cases}
\eps_\beta f_\beta & \text{on $\langle r_\beta, \eps_\beta\infty\rangle$}
\\
-\eps_\beta \tau_s(g^{-1})f_\alpha & \text{on $\langle r_\beta, -\eps_\beta\infty\rangle$}
\end{cases}
\]
is also one and $\eps_\alpha = \eps_\beta$, $r_\beta = g^{-1}.\infty$ and $g^{-1}.r_\alpha=\infty$. Thus, $\tau_s(g^{-1})\psi_\alpha = \psi_\beta$.
\end{remark}

\begin{example}
For the lattice $\Gamma$ from Example~\ref{ex_Gamma} with the choices of $\fpch$, $\choices$ and $\shmap$ done there we set
\[
 f_1 \sceq f_{(\pch_1,g),1},\ f_2 \sceq f_{(\pch_1,g),2},\ f_3\sceq f_{(\pch_1,g),3},\ \text{and}\ f_4\sceq f_{(\pch_2,g),1}.
\]
The space of period functions $\FE_s^{\omega,\dec}(\Gamma;\fpch,\choices,\shmap)$ consists of the function vectors $f=(f_1,\ldots, f_4)$ such that 
\begin{enumerate}[(i)]
\item the function $f_1$ is real-analytic on $(0,\infty)$, the function $f_2$ is real-analytic on $\big(-\infty, \tfrac13(2+\sqrt{2})\big)$, the function $f_3$ is real-analytic on $\big(-\infty, \tfrac13\big)$, and the function $f_4$ is real-analytic on $\big(\tfrac13,\infty\big)$,
\item the following system of functional equations is satisfied:
\begin{align*}
f_1 & = \tau_s(k_1^{-1})f_1 + \tau_s(k_1^{-1})f_3
\\
f_2 & = \tau_s(k_2^{-1})f_2 + \tau_s(T)f_3
\\
f_3 & = \tau_s(g)f_4 + \tau_s(g^2)f_4 + \tau_s(g^3)f_4
\\
f_4 & = f_1 + \tau_s(g^{-1})f_2,
\end{align*}
\item the maps
\[
\begin{cases}
f_1 & \text{on $(0,\infty)$}
\\
-\tau_s(T^{-1})f_2 & \text{on $(-\infty, 0)$}
\end{cases}
\quad\text{and}\quad
\begin{cases}
f_4 & \text{on $\big(\tfrac13,\infty\big)$}
\\
-f_3 & \text{on $\big(-\infty,\tfrac13\big)$}
\end{cases}
\]
extend smoothly to $P^1(\R)$.
\end{enumerate}
\end{example}

\subsection{Parabolic $1$-cohomology} 

For the proof of Theorem~A we take advantage of the characterization in \cite{BLZ_part2} of Maass cusp forms with eigenvalue $s(1-s)$ as parabolic $1$-cocycle classes with values in the semi-analytic smooth vectors of the principal series representation with spectral parameter $s$. In the following we briefly recall this characterization.

Let $s\in\C$. The space $\mc V_s^{\omega*,\infty}$ of semi-analytic smooth vectors in the line model of the principal series representation with spectral parameter $s$ is the space of functions $\varphi\colon \R\to \C$ which are smooth and extend smoothly to $P^1(\R)$ and are real-analytic on $\R\smallsetminus E$, where $E$ is a finite subset which may depend on $\varphi$. The lattice $\Gamma$ acts on $\mc V_s^{\omega*,\infty}$ via the action $\tau_s$ from \eqref{tau}.

Recall that the space of $1$-cocycles of group cohomology of $\Gamma$  with values in $\mc V_s^{\omega*,\infty}$ is
\[
 Z^1(\Gamma;\mc V_s^{\omega*,\infty}) = \{ c\colon \Gamma \to \mc V_s^{\omega*,\infty} \mid \forall\, g,h\in\Gamma\colon c_{gh} = \tau_s(h^{-1})c_g + c_h\}.
\]
We use here the notation of restricted cocycles and write $c_g\in \mc V_s^{\omega*,\infty}$ for the image $c(g)$ of $g\in\Gamma$ under $c\colon\Gamma\to\mc V_s^{\omega*,\infty}$. The space of parabolic $1$-cocycles is
\[
Z^1_\parab(\Gamma;\mc V_s^{\omega*,\infty}) = \left\{ c\in Z^1(\Gamma;\mc V_s^{\omega*,\infty}) \left\vert\ 
\begin{split}
& \forall\, p\in\Gamma\ \text{parabolic}\ \exists\, \psi\in\mc V_s^{\omega*,\infty} \colon 
\\
& c_p = \tau_s(p^{-1})\psi -\psi 
\end{split}
\right.\right\}.
\]
The spaces of $1$-coboundaries of group cohomology and of parabolic cohomology are equal. They are
\[
 B^1_\parab(\Gamma;\mc V_s^{\omega*,\infty}) = B^1(\Gamma; \mc V_s^{\omega*,\infty}) = \{ g\mapsto \tau_s(g^{-1})\psi-\psi \mid \psi\in \mc V_s^{\omega*,\infty}\}.
\]
Then the parabolic $1$-cohomology space is the quotient space
\[
 H^1_\parab(\Gamma;\mc V_s^{\omega*,\infty}) = Z^1_\parab(\Gamma;\mc V_s^{\omega*,\infty})/ B^1_\parab(\Gamma;\mc V_s^{\omega*,\infty}).
\]
Let $\MCF_s(\Gamma)$ denote the space of Maass cusp forms for $\Gamma$ with eigenvalue $s(1-s)$. 

\begin{thm}[\cite{BLZ_part2}]\label{Parab} Let $s\in\C$ with $\Rea s \in (0,1)$. Then the vector spaces $\MCF_s(\Gamma)$ and $H^1_\parab(\Gamma;\mc V_s^{\omega*,\infty})$ are isomorphic.
\end{thm}

The isomorphism in Theorem~\ref{Parab} is constructive and given by the following integral transform. Let $u\in \MCF_s(\Gamma)$ and choose any $z_0\in\h$ or cuspidal. Then the parabolic $1$-cocycle class $[c]\in H^1_\parab(\Gamma;\mc V_s^{\omega*,\infty})$ associated to $u$ is represented by the cocycle $c$ given by
\[
 c_g(r) \sceq \int_{g^{-1}.z_0}^{z_0} [u, R(r,\cdot)^s]
\]
for $g\in\Gamma$. Here we use $R\colon \R\times\h\to\h$, 
\[
 R(r,z) \sceq \Ima\left(\frac{1}{r-z}\right)
\]
and
\[
 [u,v] \sceq \frac{\partial u}{\partial z} \cdot v\cdot dz + u\cdot\frac{\partial v}{\partial\overline z}\cdot d\overline z \qquad\text{(Green form)}
\]
for any complex-valued smooth function $v$ on $\h$. The integration is performed along any differentiable path from $g^{-1}.z_0$ to $z_0$ which is essentially contained in $\h$. The integral is well-defined since the $1$-form $[u,R(r,\cdot)^s]$ is closed. A change of the choice of $z_0$ changes $c$ by a parabolic $1$-coboundary. The $\Gamma$-action via $\tau_s$ translates into a change of path of integration. More precisely, we have
\begin{equation}\label{int_tau}
 \tau_s(g^{-1})\int_a^b [u,R(r,\cdot)^s] = \int_{g^{-1}.a}^{g^{-1}.b} [u,R(r,\cdot)^s]
\end{equation}
for all $g\in\Gamma$ and cuspidal points $a,b$.

\section{Isomorphism of period functions and Maass cusp forms}\label{sec_isom}

In this section we prove the following statement: 

\begin{thm}\label{finaliso}
For $s\in\C$ with $\Rea s \in (0,1)$, the vector spaces $\FE_s^{\omega,\dec}(\Gamma;\fpch,\choices,\shmap)$ and $H^1_\parab(\Gamma;\mc V_s^{\omega*,\infty})$ are isomorphic.
\end{thm}

This, together with Theorem~\ref{Parab}, establishes Theorem~A. The isomorphism in Theorem~\ref{finaliso} is provided by the two constructions presented in the following. For these, let 
\[
 \Sigma'\sceq \{\alpha\in\Sigma \mid \text{$r_\alpha$ is $\Gamma$-equivalent to $\infty$}\}.
\]
If $\alpha\in\Sigma'$, then $C'_\alpha$ is as in Situation (1) or (3). Hence there exists $b\in\Gamma$ such that $C'_\alpha$ is based on the geodesic segment $(b^{-1}.\infty,\infty)$. The element $b$ is unique only up to left multiplication with elements in $\Gamma_\infty$. Let $(\beta,g)$ be the tuple assigned to $\alpha$ (cf.\@ Section~\ref{sec_per}). For any possible choice of $b$, we call $(\beta,g,b)$ a \textit{triple assigned to} $\alpha$.
Let
\[
\mc S \sceq \{ b\in\Gamma \mid \text{$\exists\, \alpha\in\Sigma'\ \exists\,\beta\in\Sigma\ \exists\, g\in\Gamma\colon (\beta,g,b)$ is assigned to $\alpha$}\} \cup \{T\}. 
\]

\begin{enumerate}[(a)]
\item Let $f\in \FE_s^{\omega,\dec}(\Gamma;\fpch,\choices,\shmap)$. We define a map 
\[
 c \sceq c(f) \colon \mc S \to \mc V^{\omega*,\infty}_s
\]
by $c_T \sceq 0$ and 
\[
 c_{b} \sceq \psi_{\alpha,f} \quad\text{for $\alpha\in\Sigma'$,}
\]
where $(\beta,g, b)\in\Sigma\times\Gamma\times\Gamma$ is a triple assigned to $\alpha$. Proposition~\ref{ftoc} below shows that $c$ determines a unique parabolic $1$-cocycle.
\item Let $[c] \in H^1_\parab(\Gamma;\mc V_s^{\omega*,\infty})$. Pick its unique representative $c\in Z^1_\parab(\Gamma; \mc V_s^{\omega*,\infty})$ for which $c_T = 0$. We associate to $[c]$ a function vector $f([c]) = (f_\alpha)_{\alpha\in\Sigma}$ as follows. Suppose first that $\alpha\in \Sigma'$. Pick a triple $(\beta, g, b) \in \Sigma \times \Gamma\times\Gamma$ which is assigned to $\alpha$. We define
\[
 f_\alpha \sceq \eps_\alpha c_b \cdot 1_{I_\alpha}.
\]
Suppose now that $\alpha\in\Sigma\smallsetminus \Sigma'$. Recall that $C'_\alpha$ is then based on a geodesic segment of the form $(v,\infty)$ with $v$ being a cuspidal point which is not $\Gamma$-equivalent to $\infty$. Let $p$ be a generator of $\Stab_\Gamma(v)$. Let $\psi\in \mc V_s^{\omega*,\infty}$ be the unique element (see \cite[Lemma~3.3]{Pohl_mcf_Gamma0p}) such that 
\[
 c_p = \tau_s(p^{-1})\psi - \psi.
\]
We define
\[
 f_\alpha \sceq -\eps_\alpha \psi\cdot 1_{I_\alpha}.
\]
\end{enumerate}

Instead of Theorem~\ref{finaliso} we show its following concretization.

\begin{thm}\label{actualfinal}
Let $s\in\C$ with $0 < \Rea s < 1$. Then the map
\[
 \FE_s^{\omega,\dec}(\Gamma;\fpch,\choices,\shmap) \to H^1_\parab(\Gamma;\mc V_s^{\omega*,\infty}),\quad f \mapsto [c(f)]
\]
is a linear isomorphism. Its inverse map is given by
\[
 H^1_\parab(\Gamma;\mc V_s^{\omega*,\infty}) \to \FE_s^{\omega,\dec}(\Gamma;\fpch,\choices,\shmap),\quad [c] \mapsto f([c]).
\]
\end{thm}

The proof of Theorem~\ref{actualfinal} is split into Propositions~\ref{ftoc} and \ref{ctof} below.

\begin{prop}\label{ftoc}
If $f\in \FE_s^{\omega,\dec}(\Gamma;\fpch,\choices,\shmap)$, then $c(f)$ determines a unique element in $Z^1_\parab(\Gamma;\mc V_s^{\omega*,\infty})$.
\end{prop}

\begin{remark}
In order to extend $c \sceq c(f)$ to all of $\Gamma$ and to show that this extension is well-defined, unique and a $1$-cocycle, we want to apply the Poincar\'e Fundamental Polyhedron Theorem (see e.g.\@ \cite{Maskit}) to the (closed) fundamental region $\bigcup\fpch$ for $\Gamma$. The Poincar\'e Theorem in its usual form however may only be applied if $\bigcup\fpch$ is connected. For each $\pch\in\fpch$ we can use an appropriate $T$-shift, say $T^{n(\pch)}$, such that the union over the family
\[
 \fpch'\sceq \{T^{n(\pch)}\pch \mid \pch\in\fpch\}
\]
provides a (connected, closed) fundamental domain for $\Gamma$. The side-pairing elements of the non-vertical sides of the elements in $\fpch'$ are also $T$-shifted compared to the original ones. If we apply the Poincar\'e Theorem to $\bigcup\fpch'$ to deduce the relations between the side-pairing elements, the $T$-shifts will cancel and we will find the same relations as if we had applied the Poincar\'e Theorem to $\bigcup\fpch$. In short, for our purposes we may apply the Poincar\'e Theorem to $\bigcup\fpch$ if we add $T$ to the set of generators even though it need not be a side-pairing element.
\end{remark}

\begin{proof}[Proof of Proposition~\ref{ftoc}]
Let $c\sceq c(f)$. The Poincar\'e Theorem shows that $\Gamma$ is generated (as a group) by $\mc S$. Therefore, using the definition of $c$ on $\mc S$ we can extend $c$ to all of $\Gamma$ via the cocycle relation.  If this extension (which we also call $c$) is well-defined, then it is unique. Moreover, the properties (PF3) and (PF4) in Section~\ref{sec_per} and Remark~\ref{issmooth} yield that $c$ takes values in $\mc V_s^{\omega*,\infty}$. To show that $c$ is well-defined and indeed defines a $1$-cocycle we proceed in the following steps. The Poincar\'e Theorem yields that these are sufficient.

\begin{enumerate}[(i)]
\item\label{wd1} If $b_1 = T^nb_2$ for some $b_1,b_2\in\mc S$, $n\in\Z$, then we have to show that 
\[
 c_{b_1} = \tau_s(b_2^{-1})c_{T^n} + c_{b_2}.
\]
Now $c_T=0$ implies $c_{T^n} = 0$, and by definition $c_{b_1} = c_{b_2}$. Thus, this condition is satisfied. Moreover, we may use $c_{T^ng} = c_g$ in any of the following considerations.
\item Suppose that we have $b,b^{-1}\in\mc S\smallsetminus\{T\}$ (or more general, $b,T^nb^{-1}\in\mc S$). We need to show $\tau_s(b)c_b = -c_{b^{-1}}$. There exist $\alpha,\beta\in\Sigma'$ and $g\in\Gamma$ such that $C'_\alpha, g.C'_\beta$ are both based on the geodesic segment $(b^{-1}.\infty,\infty)$ but are disjoint. Moreover, $b.C'_\alpha, bg.C'_\beta$ are the two translates of some components of $C'$ which are both based on the geodesic segment $(b.\infty,\infty)$, see Figure~\ref{inversepic}. 

\begin{figure}[h]
\begin{center}
\includegraphics{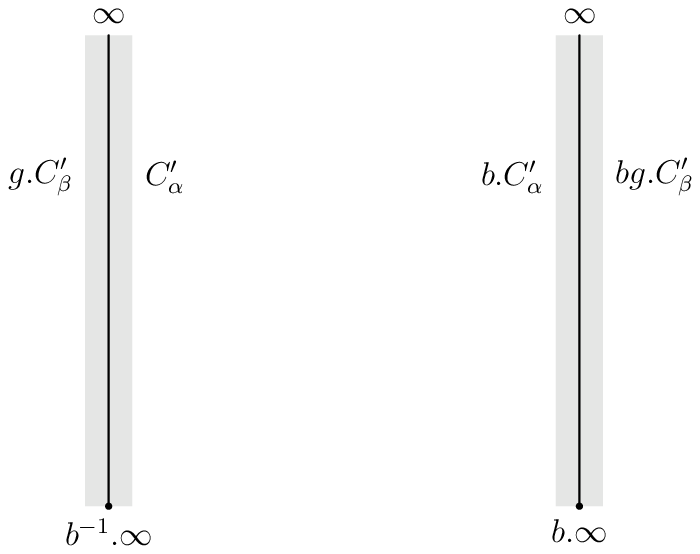}
\end{center}
\caption{Situation for $b,b^{-1}\in \mc S\smallsetminus\{T\}$}
\label{inversepic}
\end{figure}

Thus, $bg = T^n$ for some $n\in\Z$. Using \eqref{wd1} we may suppose $n=0$. Then
\[
 c_b = \psi_{\alpha, f} = 
\begin{cases}
\eps_\alpha f_\alpha & \text{on $\langle b^{-1}.\infty, \eps_\alpha\infty\rangle$}
\\
-\eps_\alpha\tau_s(b^{-1})f_\beta & \text{on $\langle b^{-1}.\infty, -\eps_\alpha\infty\rangle$}
\end{cases}
\]
and
\[
c_{b^{-1}} = \psi_{\beta,f} =
\begin{cases}
\eps_\beta f_\beta & \text{on $\langle b.\infty, \eps_\beta\infty\rangle$}
\\
-\eps_\beta \tau_s(b)f_\alpha & \text{on $\langle b.\infty, -\eps_\beta\infty\rangle$.}
\end{cases}
\]
Since $\eps_\alpha = \eps_\beta \seqc\eps$ it follows
\begin{align*}
\tau_s(b)c_b &= 
\begin{cases} 
\eps\tau_s(b) f_\alpha & \text{on $\langle b.\infty, -\eps\infty\rangle$}
\\
-\eps f_\beta & \text{on $\langle b.\infty, \eps\infty\rangle$} 
\end{cases}
\\
& = -c_{b^{-1}}.
\end{align*}
\item Let $(\pch,h)\in\choices$ with $\pch$ being a rectangle and let $\big((\pch_j,h_j)\big)_{j=1,\ldots,k}$ be the cycle determined by $(\pch,h)$. We have to show that 
\[
 c_{h_{j+k-1}\cdots h_{j+1} h_{j}} = 0
\]
for some $j\in\{1,\ldots, k\}$, where we understand the indices of the $h_\ell$ modulo $k$. 
Let $\ell\in\{1,\ldots, k\}$. Then $h_\ell\in\mc S$. Let $\alpha_\ell\sceq \big( (\pch,h),\ell\big)$ and $(\beta_\ell,g_\ell)\in\Sigma\times\Gamma$ be the tuple assigned to $\alpha_\ell$. Further set $\eps_\ell\sceq \eps_{\alpha_\ell}$ and $f_\ell\sceq f_{\alpha_\ell}$. Then 
\[
c_{h_\ell} =
\begin{cases}
\eps_\ell f_\ell & \text{on $\langle h_\ell^{-1}.\infty,\eps_\ell\infty\rangle$}
\\
-\eps_\ell\tau_s(g_\ell)f_{\beta_\ell} & \text{on $\langle h_\ell^{-1}.\infty, -\eps_\ell\infty\rangle$.}
\end{cases}
\]
By definition, we have
\[
 c_{h_{j+k-1}\cdots h_j} = \sum_{\ell=0}^{k-1} \tau_s(h_j^{-1}\cdots h_{j+\ell-1}^{-1}) c_{h_{j+\ell}}.
\]
Therefore we have to show
\begin{equation}\label{toshow}
c_{h_j} = - \sum_{\ell=0}^{k-2} \tau_s(h_j^{-1}\cdots h_{j+\ell}^{-1}) c_{h_{j+\ell+1}}.
\end{equation}
\textbf{Preliminary step:} Suppose that we are in Situation (1a). We claim that \eqref{toshow} is satisfied on $\langle h_{j+k-1}.\infty, \eps_j\infty\rangle$. There exist $j\in\{1,\ldots,k\}$ and $n_j\in\Z$ such that $g_{j+k-1} = h_{j+k-2}\cdots h_jT^{-n_j}$ (for Figure~\ref{sit1a} this means that $\alpha = \beta_{j+k-1}$ and $n=n_j$). Note that 
\[
 h_j^{-1}\cdots h_{j+k-2}^{-1} = h_{j+k-1}
\]
and $g_{j+k-1}=h_{j+k-1}^{-1}T^{-n_j} \in \mc S$. Further note that $\eps_\ell = \eps_{\beta_{j+k-1}}\seqc \eps$ for all $\ell=1,\ldots,k$. Then 
\[
c_{h_{j+k-1}^{-1}T^{-n_j}} = 
\begin{cases}
\eps f_{\beta_{j+k-1}} & \text{on $\langle T^{n_j}h_{j+k-1}.\infty,\eps\infty\rangle$}
\\
-\eps \tau_s(T^{n_j}h_{j+k-1})f_{j+k-1} & \text{on $\langle T^{n_j}h_{j+k-1}.\infty, -\eps\infty\rangle$.}
\end{cases}
\]
From $f$ being a $1$-eigenfunction of $\mc L_{F,s}$ it follows that 
\[
 f_{\beta_{j+k-1}} = \sum_{\ell=0}^{k-2} \tau_s(T^{n_j}h_j^{-1}\cdots h_{j+\ell-1}^{-1}) f_{j+\ell}.
\]
On $\langle T^{n_j}h_{j+k-1}.\infty,\eps\infty\rangle$ we have
\begin{align*}
c_{h_{j+k-1}^{-1}T^{-n_j}} & = \eps f_{\beta_{j+k-1}} = \sum_{\ell=0}^{k-2} \tau_s(T^{n_j}h_j^{-1}\cdots h_{j+\ell-1}^{-1}) \eps f_{j+\ell}
\\
& = \sum_{\ell=0}^{k-2} \tau_s(T^{n_j}h_j^{-1}\cdots h_{j+\ell-1}^{-1}) c_{h_{j+\ell}}.
\end{align*}
Now
\[
 c_{h_{j+k-1}^{-1}T^{-n_j}} = -\tau_s(T^{n_j}) \tau_s(h_{j+k-1})c_{h_{j+k-1}}
\]
yields
\[
 0 = \sum_{\ell=0}^{k-1} \tau_s(h_j^{-1}\cdots h_{j+\ell-1}^{-1}) c_{h_{j+\ell}} \quad\text{on $\langle h_{j+k-1}.\infty, \eps\infty\rangle$.}
\]
\textbf{Main step:} Suppose now that we are in Situation (1b), hence there exist $j\in\{1,\ldots,k\}$ and $n_j\in\Z$ such that $g_j = T^{-n_j}$ (for Figure~\ref{sit1b} this means that $\alpha = \beta_j$ and $n=n_j$). Then $h_jT^{-n_j} \in \mc S$ and
\[
 c_{h_jT^{-n_j}} =
\begin{cases}
\eps_{\beta_j} f_{\beta_j} & \text{on $\langle T^{n_j}h_j^{-1}.\infty, \eps_{\beta_j}.\infty\rangle$}
\\
-\eps_{\beta_j}\tau_s(T^{n_j})f_j & \text{on $\langle T^{n_j}h_j^{-1}.\infty, -\eps_{\beta_j}.\infty\rangle$.}
\end{cases}
\]
From $f$ being a $1$-eigenfunction of $\mc L_{F,s}$ it follows that 
\[
 f_{\beta_j} = \sum_{\ell=0}^{k-2}\tau_s(T^{n_j}h_j^{-1}\cdots h_{j+\ell}^{-1})f_{j+\ell+1}.
\]
Note that $\eps_{\beta_j} = -\eps_\ell \seqc \eps'$ for $\ell=1,\ldots, k$. On $\langle T^{n_j}h_j^{-1}.\infty, \eps'\infty\rangle$ we have
\begin{align*}
c_{h_jT^{-n_j}} & = \eps' f_{\beta_j} = -\sum_{\ell=0}^{k-2} \tau_s(T^{n_j}h_j^{-1}\cdots h_{j+\ell}^{-1})\eps_{j+\ell+1}f_{j+\ell+1}
\\
& = -\sum_{\ell=0}^{k-2}\tau_s(T^{n_j}h_j^{-1}\cdots h_{j+\ell}^{-1})c_{h_{j+\ell+1}}.
\end{align*}
Since 
\[
 c_{h_jT^{-n_j}} = \tau_s(T^{n_j})c_{h_j},
\]
the equality \eqref{toshow} is verified on $\langle h_j^{-1}.\infty, \eps'\infty\rangle$. We now show that \eqref{toshow} holds on 
\[
 h_j^{-1}.\langle h_{j+1}^{-1}.\infty,\eps'\infty\rangle = 
\begin{cases}
(h_j^{-1}h_{j+1}^{-1}.\infty,h_j^{-1}.\infty) & \text{if $\eps'=+1$}
\\
(h_j^{-1}.\infty, h_j^{-1}h_{j+1}^{-1}.\infty) & \text{if $\eps'=-1$.}
\end{cases}
\]
If $g_{j+1} = T^{-n_{j+1}}$ for some $n_{j+1}\in\Z$, then we process as before to see that 
\begin{align*}
c_{h_{j+1}} & = - \sum_{\ell=0}^{k-2} \tau_s(h_{j+1}^{-1}\cdots h_{j+\ell+1}^{-1})c_{h_{j+\ell+2}}
\\
& = - \tau_s(h_{j+1}^{-1}\cdots h_{j+k-1}^{-1})c_{h_j} - \sum_{\ell=1}^{k-2} \tau_s(h_{j+1}^{-1}\cdots h_{j+\ell}^{-1})c_{h_{j+\ell+1}}
\end{align*}
on $\langle h_{j+1}^{-1}.\infty, \eps'\infty\rangle$. Since
\[
 h_j = h_{j+1}^{-1}\cdots h_{j+k-1}^{-1},
\]
an application of $\tau_s(h_j^{-1})$ on both sides shows \eqref{toshow} on $h_j^{-1}.\langle h_{j+1}^{-1}.\infty,\eps'\infty\rangle$.

We now consider the case that $g_{j+1} \notin \Gamma_\infty$. By Situations (1a) and (3b), 
\[
 g_{j+1}.\infty = h_{j+1}^{-1}.\infty
\]
and hence $g_{j+1} = h_{j+1}^{-1}T^m$ for some $m\in\Z$. This means we are in Situation (1a) with $j+2$ instead of $j$ and $-m$ instead of $n_j$. The preliminary step from above shows that on $\langle h_{j+1}.\infty, -\eps'\infty\rangle$ we have
\begin{align*}
0 & = \sum_{\ell=0}^{k-1} \tau_s(h_{j+2}^{-1}\cdots h_{j+\ell+1}^{-1}) c_{h_{j+\ell+2}}
\\
& = \sum_{\ell=0}^{k-2} \tau_s(h_{j+2}^{-1}\cdots h_{j+\ell+1}^{-1}) c_{h_{j+\ell+2}} + \tau_s(h_{j+2}^{-1}\cdots h_{j+k}^{-1})c_{h_{j+1}}.
\end{align*}
Using $h_{j+2}^{-1}\cdots h_{j+k}^{-1} = h_{j+1}$ and applying $\tau_s(h_{j+1}^{-1})$ we find
\begin{align*}
0 &= \sum_{\ell=1}^{k-1} \tau_s(h_{j+1}^{-1} h_{j+2}^{-1}\cdots h_{j+\ell}^{-1})c_{h_{j+\ell+1}} + c_{h_{j+1}}
\\
& = \sum_{\ell=0}^{k-2} \tau_s(h_{j+1}^{-1}\cdots h_{j+\ell}^{-1}) c_{h_{j+\ell+1}} + \tau_s(h_{j+1}^{-1}\cdots h_{j+k-1}^{-1})c_{h_j}
\end{align*}
on $h_{j+1}^{-1}.\langle h_{j+1}.\infty, -\eps'\infty\rangle = \langle h_{j+1}^{-1}.\infty, \eps'\infty\rangle$. Applying $\tau_s(h_j^{-1})$ shows \eqref{toshow} on $h_j^{-1}.\langle h_{j+1}^{-1}.\infty,\eps'\infty\rangle$.

Iterating this argument, eventually establishes \eqref{toshow} on $P^1(\R)$ if we start in Situation (1b). The proof for a start in Situation (1a) is analogous.

\item\label{wd4} Suppose that we have an order $2$ relation of the form $(T^mh)^2 = \id$ constructed by the Poincar\'e Theorem from a situation as indicated in Figure~\ref{pic_wd4}, where $\pch\in\fpch$ is a rectangle, $h\in \Gamma$ a side-pairing element and $m\in\Z$. 

\begin{figure}[h]
\begin{center}
\includegraphics{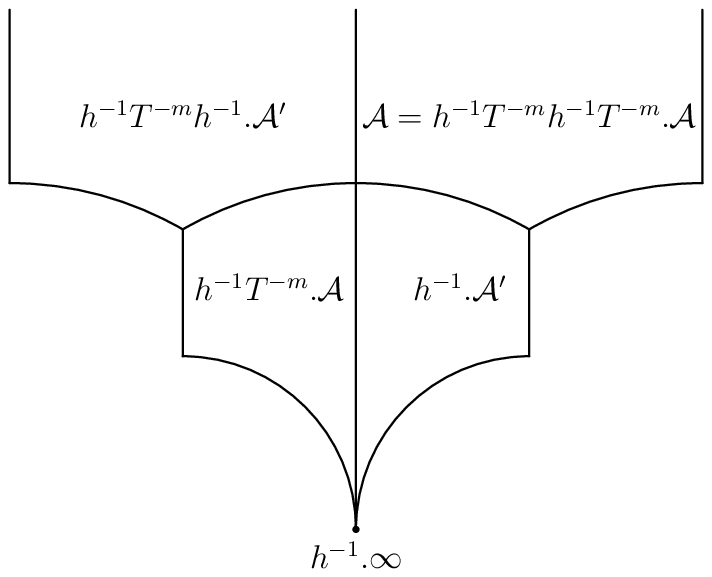}
\end{center}
\caption{Order $2$ relation with $\pch$ being a rectangle}
\label{pic_wd4}
\end{figure}

We have to show 
\begin{equation}\label{toshow2}
 c_{(T^mh)^2} = 0.
\end{equation}

Suppose first that $(\pch, h)$ is an element of a cycle determined by an element in $\choices$. Without loss of generality, we may assume that $(\pch, h) \in \choices$. Then $h\in \mc S$. Let $\alpha\sceq \big((\pch,h),1\big)$. Then $C'_\alpha$ and $h^{-1}T^{-m}C'_\alpha$ are both based on the geodesic segment $(h^{-1}.\infty,\infty)$ and are disjoint. Thus
\[
 c_h = 
\begin{cases}
\eps_\alpha f_\alpha & \text{on $\langle h^{-1}.\infty, \eps_\alpha\infty\rangle$}
\\
-\eps_\alpha \tau_s(h^{-1}T^{-m})f_\alpha & \text{on $\langle h^{-1}.\infty, -\eps_\alpha\infty\rangle$.}
\end{cases}
\]
By definition and since $c_T=0$ we have
\[
c_{(T^mh)^2} = \tau_s(h^{-1}T^{-m})c_h + c_h.
\]
Now an easy calculation establishes \eqref{toshow2}.

Suppose now that $(\pch,h)$ is not an element of a cycle determined by some element in $\choices$. Then, for some $\pch'\in\fpch$, the pair $(\pch',h^{-1})$ appears in a cycle determined by some element in $\choices$ and we can apply the previous argument to show
\[
 c_{(h^{-1}T^{-m})^2} = 0.
\]
\item Suppose that we have an order $2$ relation of the form $(T^nh)^2 = \id$ constructed by the Poincar\'e Theorem from a situation as indicated in Figure~\ref{pic_wd5}, where $\pch\in\fpch$ is a triangle, $h\in \Gamma$ a side-pairing element and $n\in\Z$. Here one shows along the lines of \eqref{wd4} that $c$ vanishes on $(T^nh)^2$.

\begin{figure}[h]
\begin{center}
\includegraphics{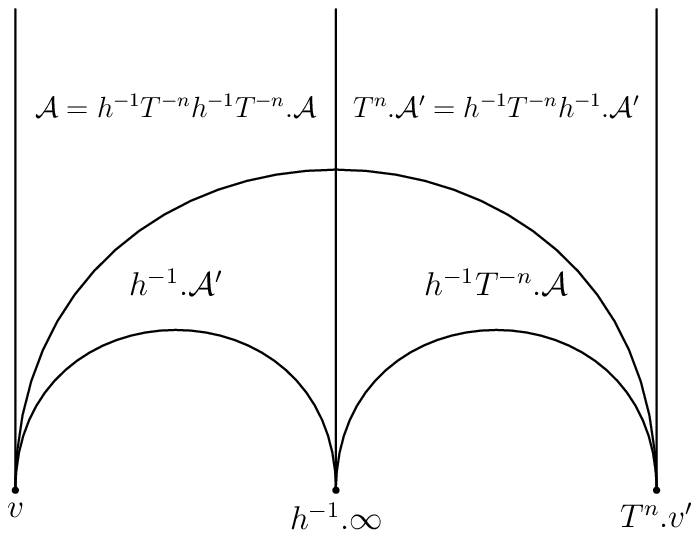}
\end{center}
\caption{Order $2$ relation with $\pch$ being a triangle}
\label{pic_wd5}
\end{figure}

\end{enumerate}
This shows that $c$ is a $1$-cocycle. It remains to prove that $c$ is parabolic. Let $(\pch,h)\in\choices$ with $\pch$ being a triangle. Let $v$ be the infinite vertex of $\mc K$ to which $\pch$ is associated. By the Poincar\'e Theorem there is a finite sequence in $\fpch\times\Gamma$ determining a generator of $\Stab_\Gamma(v)$. This sequence is constructed as follows. Let $\pch_1 \sceq \pch$, $h_1\sceq h$ and let $\big( (\pch_1,h_1), (\pch'_1, h_1^{-1})\big)$ be the cycle determined by $(\pch_1,h_1)$. Iteratively for $j=2,3,\ldots$ we find $n_{j-1}\in\Z$ and $(\pch_j,h_j)\in \fpch\times\Gamma$, $\pch_j$ a triangle with non-vertical side in $I(h_j)$ such that $T^{n_{j-1}}\pch_j$ coincides with $\pch'_{j-1}$ on their long side (and only there). Let $\pch'_j$ be the element in $\fpch$ such that $\big( (\pch_j,h_j), (\pch'_j, h_j^{-1})\big)$ is the cycle determined by $(\pch_j,h_j)$. This construction stops when there exists $n_j\in\Z$ such that $T^{n_j}\pch'_j$ coincides with $\pch_1$ precisely on their long 
side. Then 
\[
 p\sceq T^{n_j}h_jT^{-n_{j-1}}h_{j-1}T^{-n_{j-2}}\cdots h_1
\]
is a generator of $\Stab_\Gamma(v)$ and as such a parabolic element. Let $C'_1$ be the component of $C'$ which is based on the long side of $\pch_1$ and whose elements point into $\pch_1$, and let $C'_{n_j}$ be the component which is based on the long side of $\pch_{n_j}$ and whose elements point into $\pch_{n_j}$. Further let $f_1, f_{n_j}$ denote the corresponding components of $f$. Define 
\[
 \varphi\sceq 
\begin{cases}
-\eps f_1 & \text{on $\langle v, \eps\infty\rangle$}
\\
\eps\tau_s(T^{n_j})f_{n_j} & \text{on $\langle v, -\eps\infty\rangle$.}
\end{cases}
\]
By (PF4), $\varphi$ extends smoothly to $P^1(\R)$. Using $f=\mc L_{F,s}f$ one proves in analogy to the arguments above that 
\[
 c_p = \tau_s(p^{-1})\varphi - \varphi.
\]
Note that each cusp of $\Gamma\backslash\h$ is represented by some of these infinite vertices.
To illustrate these arguments we perform them for the situation shown in Figures~\ref{parabsett}-\ref{parabglue}. The Figure~\ref{parabhelp} is an intermediate step to construct the lower triangle with vertices $v_1, h_1^{-1}T^{n_1}h_2^{-1}.\infty, h_1^{-1}.\infty$ in Figure~\ref{parabglue}. We assume here that $(\pch_1,h_1), (\pch_2,h_2)\in\choices$, let $m_1\sceq m(\pch_1,h_1)$, $m_2\sceq m(\pch_2,h_2)$ and use obvious abbreviations for the notation. The necessary steps for the proof and their validity can actually be read off these figures. The general situation is proved in exactly the same way by iterating this specific case sufficiently often. In our case, the stabilizer of $v_1$ is generated by 
\[
 T^{n_2}h_2T^{-n_1}h_1.
\]
\begin{figure}[h]
\begin{center}
\includegraphics{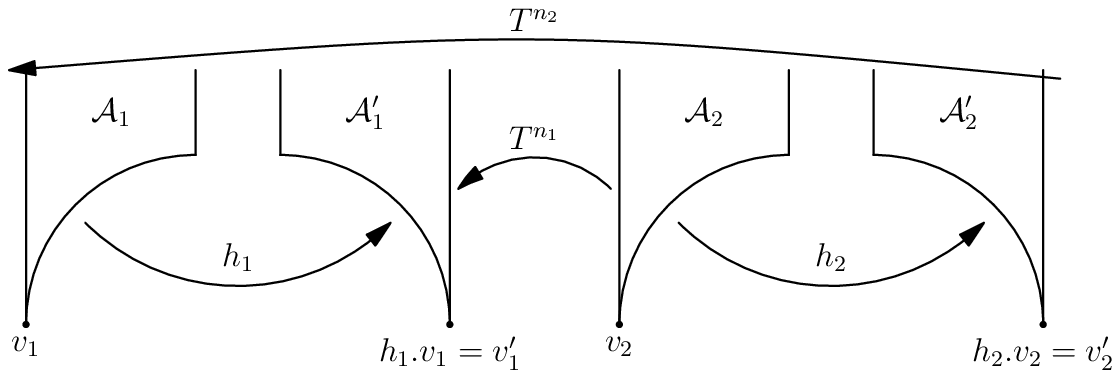}
\end{center}
\caption{Construction of ``cycle'' at $v_1$}
\label{parabsett}
\end{figure}

\begin{figure}[h]
\begin{center}
\includegraphics{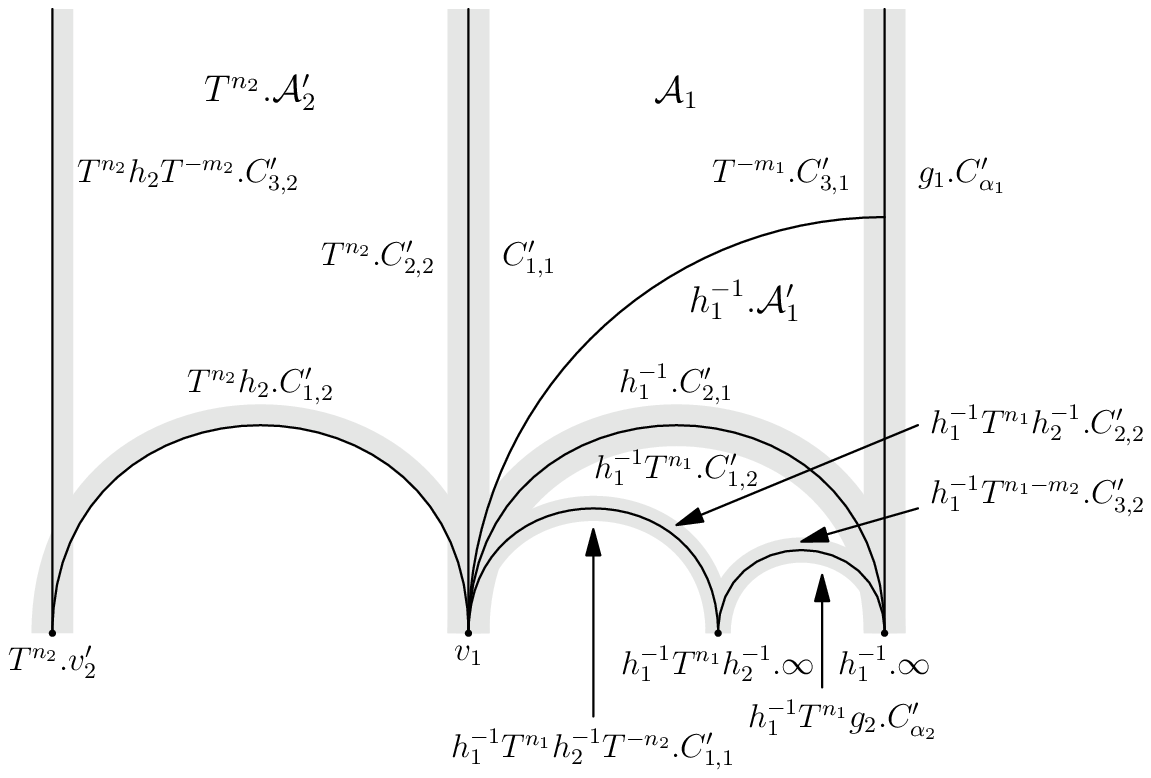}
\end{center}
\caption{Situation locally at $v_1$}
\label{parabglue}
\end{figure}

\begin{figure}[h]
\begin{center}
\includegraphics{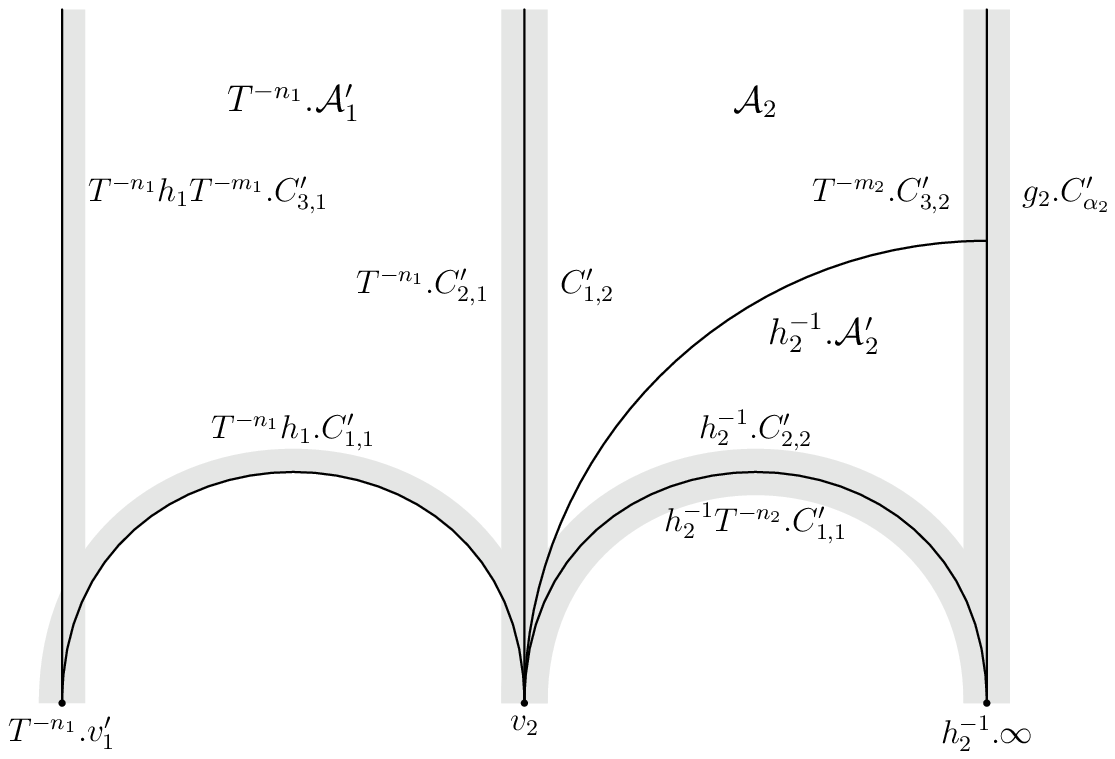}
\end{center}
\caption{Intermediate step in construction}
\label{parabhelp}
\end{figure}

We set
\[
\psi\sceq
\begin{cases}
-f_{1,1} & \text{on $(v_1,\infty)$}
\\
\tau_s(T^{n_2})f_{2,2} & \text{on $(-\infty, v_1)$}
\end{cases}
\]
and claim that 
\[
 c_{T^{n_2}h_2T^{-n_1}h_1} = \tau_s(h_1^{-1}T^{n_1}h_2^{-1}T^{-n_2})\psi - \psi.
\]
From the cocycle relation and $c_T = 0$ it follows that
\[
 c_{T^{n_2}h_2T^{-n_1}h_1} = \tau_s(h_1^{-1}T^{n_1})c_{h_2} + c_{h_1}.
\]
Since 
\[
 c_{h_1} = 
\begin{cases}
 \tau_s(g_1)f_{\alpha_1} & \text{on $(h_1^{-1}.\infty,\infty)$}
\\
-\tau_s(T^{-m_1})f_{3,1} & \text{on $(-\infty, h_1^{-1}.\infty)$}
\end{cases}
\]
and
\[
c_{h_2} =
\begin{cases}
\tau_s(g_2)f_{\alpha_2} & \text{on $(h_2^{-1}.\infty,\infty)$}
\\
-\tau_s(T^{-m_2})f_{3,2} & \text{on $(-\infty,h_2^{-1}.\infty)$}
\end{cases}
\]
we have
\[
 \tau_s(h_1^{-1}T^{n_1})c_{h_2} = 
\begin{cases}
\tau_s(h_1^{-1}T^{n_1}g_2)f_{\alpha_2} & \text{on $(h_1^{-1}T^{n_1}h_2^{-1}.\infty, h_1^{-1}.\infty)$}
\\
-\tau_s(h_1^{-1}T^{n_1-m_2})f_{3,2} & \text{on $(-\infty, h_1^{-1}T^{n_1}h_2^{-1}.\infty) \cup (h_1^{-1}.\infty,\infty)$}
\end{cases}
\]
and hence
\begin{align*}
&\tau_s(h_1^{-1}T^{n_1})c_{h_2} + c_{h_1} 
\\
& \quad = 
\begin{cases}
-\tau_s(h_1^{-1}T^{n_1-m_2})f_{3,2} - \tau_s(T^{-m_1})f_{3,1} & \text{on $(-\infty, h_1^{-1}T^{n_1}h_2^{-1}.\infty)$}
\\
\tau_s(h_1^{-1}T^{n_1}g_2)f_{\alpha_2} - \tau_s(T^{-m_1})f_{3,1} & \text{on $(h_1^{-1}T^{n_1}h_2^{-1}.\infty, h_1^{-1}.\infty)$}
\\
-\tau_s(h_1^{-1}T^{n_1-m_2})f_{3,2} + \tau_s(g_1)f_{\alpha_1} & \text{on $(h_1^{-1}.\infty,\infty)$.}
\end{cases}
\end{align*}
On the other side we have
\begin{align*}
&\tau_s(h_1^{-1}T^{n_1}h_2^{-1}T^{-n_2})\psi -\psi 
\\
& \quad =
\begin{cases}
\tau_s(h_1^{-1}T^{n_1}h_2^{-1})f_{2,2} - \tau_s(T^{n_2})f_{2,2} & \text{on $(-\infty, v_1)$}
\\
-\tau_s(h_1^{-1}T^{n_1}h_2^{-1}T^{-n_2})f_{1,1} + f_{1,1} & \text{on $(v_1,h_1^{-1}T^{n_1}h_2^{-1}.\infty)$}
\\
\tau_s(h_1^{-1}T^{n_1}h_2^{-1})f_{2,2} + f_{1,1} & \text{on $(h_1^{-1}T^{n_1}h_2^{-1}.\infty,\infty)$.}
\end{cases}
\end{align*}
Therefore we have to show:

\begin{enumerate}[(i)]
\item\label{tri1} on $(-\infty,v_1)$:
\[
 \tau_s(h_1^{-1}T^{n_1}h_2^{-1})f_{2,2} - \tau_s(T^{n_2})f_{2,2} = -\tau_s(h_1^{-1}T^{n_1-m_2})f_{3,2} - \tau_s(T^{-m_1})f_{3,1},
\]
\item\label{tri2} on $(v_1,h_1^{-1}T^{n_1}h_2^{-1}.\infty)$:
\[
-\tau_s(h_1^{-1}T^{n_1}h_2^{-1}T^{-n_2})f_{1,1} + f_{1,1} = -\tau_s(h_1^{-1}T^{n_1-m_2})f_{3,2} - \tau_s(T^{-m_1})f_{3,1},
\]
\item\label{tri3} on $(h_1^{-1}T^{n_1}h_2^{-1}.\infty, h_1^{-1}.\infty)$:
\[
 \tau_s(h_1^{-1}T^{n_1}h_2^{-1})f_{2,2} + f_{1,1} = \tau_s(h_1^{-1}T^{n_1}g_2)f_{\alpha_2} - \tau_s(T^{-m_1})f_{3,1},
\]
\item\label{tri4} on $(h_1^{-1}.\infty,\infty)$:
\[
 \tau_s(h_1^{-1}T^{n_1}h_2^{-1})f_{2,2} + f_{1,1} = -\tau_s(h_1^{-1}T^{n_1-m_2})f_{3,2} + \tau_s(g_1)f_{\alpha_1}.
\]
\end{enumerate}
From $f=\mc L_{F,s}f$ it follows
\begin{align}
\label{eigena}\tag{a} f_{2,2} &= \tau_s(T^{-n_2}h_1^{-1})f_{2,1} + \tau_s(T^{n_2-m_1})f_{3,1} && \text{on $(-\infty, v'_2) = (-\infty, h_2.v_2)$,}
\intertext{and}
\label{eigenb}\tag{b} f_{2,1} &= \tau_s(T^{n_1}h_2^{-1}) f_{2,2} + \tau_s(T^{n_1-m_2})f_{3,2} && \text{on $(-\infty, v'_1) = (-\infty, T^{n_1}.v_2)$.}
\end{align}
Then \eqref{eigenb} is equivalent to 
\[
 \tau_s(T^{-n_2}h_1^{-1})f_{2,1} = \tau_s(T^{-n_2}h_1^{-1}T^{n_1}h_2^{-1})f_{2,2} + \tau_s(T^{-n_2}h_1^{-1}T^{n_1-m_2})f_{3,2}
\]
on $(-\infty, T^{-n_2}.v_1) \cup (h_1^{-1}.\infty,\infty)$ since $h_1^{-1}T^{n_1}.v_2 = v_1$. Plugging this into \eqref{eigena} gives
\[
 f_{2,2} = \tau_s(T^{-n_2}h_1^{-1}T^{n_1}h_2^{-1})f_{2,2} + \tau_s(T^{-n_2}h_1^{-1}T^{n_1-m_2})f_{3,2} + \tau_s(T^{n_2-m_1})f_{3,1}
\]
on $(-\infty, T^{-n_2}h_1^{-1}T^{n_1}.v_2)$. Since $h_1^{-1}T^{n_1}.v_2 = v_1$, this proves \eqref{tri1}.

For the proof of \eqref{tri2} we use that $f=\mc L_{F,s}f$ provides
\begin{align}
\label{eigenc}\tag{c} f_{1,1} &= \tau_s(T^{n_2}h_2)f_{1,2} + \tau_s(T^{n_2}h_2T^{-m_2})f_{3,2} && \text{on $(v_1,\infty)$,} \intertext{and}
\label{eigend}\tag{d} f_{1,2} &= \tau_s(T^{-n_1}h_1)f_{1,1} + \tau_s(T^{-n_1}h_1T^{-m_1})f_{3,1} && \text{on $(v_2,\infty)$.}
\end{align}
Then \eqref{eigenc} is equivalent to 
\[
 \tau_s(h_1^{-1}T^{n_1}h_2^{-1}T^{-n_2})f_{1,1} = \tau_s(h_1^{-1}T^{n_1})f_{1,2} + \tau_s(h_1^{-1}T^{n_1-m_2})f_{3,2}
\]
on $h_1^{-1}T^{n_1}h_2^{-1}T^{-n_2}.(v_1,\infty) = (v_1, h_1^{-1}T^{n_1}h_2^{-1}.\infty)$. Equality \eqref{eigend} is equivalent to 
\begin{equation}\label{eigend'}\tag{d'}
 \tau_s(h_1^{-1}T^{n_1})f_{1,2} = f_{1,1} + \tau_s(T^{-m_1})f_{3,1}
\end{equation}
on $(v_1,h_1^{-1}.\infty)$. Combining these two equalities shows \eqref{tri2}. 

To prove \eqref{tri3} we note that $f=\mc L_{F,s}f$ yields
\begin{equation}\label{eigene}\tag{e}
 \tau_s(g_2) f_{\alpha_2} = \tau_s(h_2^{-1})f_{2,2} + f_{1,2}
\end{equation}
on $(h_2^{-1}.\infty,\infty)$. This is equivalent to 
\[
 \tau_s(h_1^{-1}T^{n_1}g_2)f_{\alpha_2} = \tau_s(h_1^{-1}T^{n_1}h_2^{-1})f_{2,2} + \tau_s(h_1^{-1}T^{n_1})f_{1,2}
\]
on $(h_1^{-1}T^{n_1}h_2^{-1}.\infty,h_1^{-1}.\infty)$. Combining this with \eqref{eigend'} proves \eqref{tri3}.

For the proof of \eqref{tri4} we remark that $f=\mc L_{F,s}f$ shows
\begin{equation}\label{eigenf}\tag{f}
\tau_s(g_1)f_{\alpha_1} = f_{1,1} + \tau_s(h_1^{-1})f_{2,1}
\end{equation}
on $(h_1^{-1}.\infty,\infty)$. This together with \eqref{eigenb} proves \eqref{tri4}. With this the proof is finally complete.
\end{proof}

\begin{prop}\label{ctof}
If $[c]\in H^1_\parab(\Gamma;\mc V_s^{\omega*,\infty})$, then $f([c])\in \FE_s^{\omega,\dec}(\Gamma;\fpch,\choices,\shmap)$.
\end{prop}

\begin{proof}
Let $f\sceq f([c]) = (f_\alpha)_{\alpha\in\Sigma}$. Then $f$ is obviously well-defined. We start by establishing the regularity conditions on $f$. Let $c\in Z^1_\parab(\Gamma;\mc V_s^{\omega*,\infty})$ be the representative of $[c]$ with $c_T=0$. Then there exists a unique Maass cusp form $u$ with eigenvalue $s(1-s)$ such that 
\[
 c_g(r) = \int_{g^{-1}.\infty}^\infty [u, R(r,\cdot)^s]
\]
for all $g\in\Gamma$. If $v$ is a cuspidal point such that $C'_\alpha$ is based on the geodesic segment $(v,\infty)$ and $p$ is a generator of $\Stab_\Gamma(v)$, then \cite[Lemma~3.3]{Pohl_mcf_Gamma0p} shows that 
\[
 \psi(r) = - \int_v^\infty [u, R(r,\cdot)^s],\qquad r\in\R,
\]
determines the unique element $\psi$ in $\mc V_s^{\omega*,\infty}$ such that 
\[
 c_p = \tau_s(p^{-1})\psi - \psi.
\]
By the definition of $f$, for each $\alpha\in\Sigma$ we have
\begin{equation}\label{falpha}
 f_\alpha(r) = \pm \eps_\alpha \int_{r_\alpha}^\infty [u, R(r,\cdot)^s], \qquad r\in \langle r_\alpha, \eps_\alpha\infty\rangle,
\end{equation}
where the sign in front of the integral depends on whether or not $\alpha\in\Sigma'$ and does not matter here. By \cite[Lemma~3.2]{Pohl_mcf_Gamma0p}, $f_\alpha$ is real-analytic for each $\alpha\in\Sigma$. Thus, (PF1) holds. Moreover, 
\[
 \psi_{\alpha,f}(r) = \int_{r_\alpha}^\infty [u, R(r,\cdot)^s] \qquad\text{for $r\in\R\setminus\{r_\alpha\}$.}
\]
These functions clearly satisfy (PF3) and (PF4). To see (PF2), we use that the $1$-form $[u, R(r,\cdot)^s]$ is closed (for all $r$) and hence we may change the path of integration. For $\alpha\in\Sigma$ let $\gamma_\alpha$ denote the geodesic from $r_\alpha$ to $\infty$. If we use as path of integration in \eqref{falpha} instead of $\gamma_\alpha$ the sequence $g_1.\gamma_{\beta_1}$, \ldots, $g_\ell.\gamma_{\beta_\ell}$, where $g_j\in\Gamma$, $\beta_j\in\Sigma$ are as indicated in Figures~\ref{sit1a}-\ref{sit3ab}, then with \eqref{int_tau} we have
\begin{align*}
 f_\alpha(r) &= \int_{\gamma_\alpha} [u, R(r,\cdot)^s] = \sum_{j=1}^\ell \int_{g_j.\gamma_{\beta_j}} [u, R(r,\cdot)^s]= \sum_{j=1}^\ell \tau_s(g_j)\int_{\gamma_{\beta_j}}[u, R(r,\cdot)^s]
\\ & = \sum_{j=1}^\ell \tau_s(g_j)f_{\beta_j}(r).
\end{align*}
This completes the proof.
\end{proof}

\section{The effect of different choices}\label{sec_choices}

The definition of period functions in Section~\ref{sec_per} is subject to the choices of $\fpch,\choices$ and $\shmap$. Let $\fpch, \choices, \shmap$ and $\wt\fpch, \wt\choices, \wt\shmap$ be two such choices. By Theorem~\ref{finaliso} (or Theorem~A  or \ref{actualfinal}) the spaces $\FE_s^{\omega,\dec}(\Gamma;\fpch,\choices,\shmap)$ and $\FE_s^{\omega,\dec}(\Gamma; \wt\fpch,\wt\choices,\wt\shmap)$ are isomorphic, and their canonical isomorphism can be determined by composing their isomorphisms with $H^1_\parab(\Gamma;\mc V_s^{\omega*,\dec})$ from Theorem~\ref{actualfinal}. 

However, the geometric background of the definition of period functions shows that there is a direct approach to this isomorphism. Let $C'$ resp.\@ $\wt C'$ denote the set of representatives for the cross section $\wh C$ which is associated to $\fpch,\choices,\shmap$ resp.\@ $\wt\fpch,\wt\choices,\wt\shmap$, and let $\Sigma$ resp.\@ $\wt\Sigma$ be the arising set of symbols. For each component $C'_\alpha$, $\alpha\in\Sigma$, of $C'$ there is a unique symbol $\wt\alpha \in \wt\Sigma$ and an element $g_\alpha\in\Gamma$ such that $g_\alpha.C'_\alpha$ equals the component $\wt C'_{\wt\alpha}$ of $\wt C'$. This relation directly translates to the level of period functions as stated in the following proposition. The proof of this proposition is straightforward.

\begin{prop}\label{isoper}
The map 
\[
 \FE_s^{\omega,\dec}(\Gamma;\fpch,\choices,\shmap) \to \FE_s^{\omega,\dec}(\Gamma;\wt\fpch,\wt\choices,\wt\shmap),\quad f=(f_\alpha)_{\alpha\in\Sigma} \mapsto \wt f = (\wt f_{\wt\alpha})_{\wt\alpha\in\wt\Sigma},
\]
where
\[
 f_{\wt\alpha} \sceq \tau_s(g_\alpha)f_{\alpha},
\]
is an isomorphism of vector spaces. It coincides with the composition of the isomorphisms from Theorem~\ref{actualfinal}.
\end{prop}

The isomorphism in Proposition~\ref{isoper} also shows how the associated transfer operator families transform into each other.  

\begin{remark}
In the following we provide an algorithm to calculate the map $\ \wt{ }\ \colon \Sigma \to \wt\Sigma$ and the elements $g_{\alpha}$ for $\alpha \in\Sigma$. 

For each $\pch\in\fpch$ there are unique elements $n(\pch)\in\Z$ and $\wt\pch\in\wt\fpch$ such that
\[
 T^{n(\pch)}\pch = \wt\pch.
\]
If $h\in\Gamma$ is the side-pairing element (in $\fpch$) which maps the side $b_1$ of $\pch_1$ to the side $b_2$ of $\pch_2$ for $\pch_1,\pch_2\in\fpch$, then 
\[
 \wt h \sceq T^{n(\pch_2)} h T^{-n(\pch_1)}
\]
is the side-pairing element (in $\wt\fpch$) which maps the side $\wt b_1 = T^{n(\pch_1)}b_1$ of $\wt\pch_1$ to the side $\wt b_2$ of $\wt\pch_2$. 

Let $(\pch, h) \in \choices$. Suppose first that $\pch$ is a triangle and let $m\sceq m(\pch, h)$. Recall that the cycle in $\fpch\times\Gamma$ determined by $(\pch, h)$ is 
\[
 \big( (\pch, h), (\pch', h^{-1}) \big)
\]
with a unique element $\pch'\in \fpch$. Now $\wt\choices$ contains a unique generator of the equivalence class of cycles in $\wt\fpch\times\Gamma$ determined by $(\wt\pch, \wt h)$. This generator is either $(\wt\pch, \wt h)$ or $(\wt\pch', \wt h^{-1})$.

If $(\wt\pch,\wt h)\in \wt\choices$, then 
\[
 \wt C'_{(\wt\pch, \wt h), 1}  = T^{n(\pch)}. C'_{(\pch, h), 1},\quad \wt C'_{(\wt\pch,\wt h), 2}  = T^{n(\pch')}. C'_{(\pch, h), 2}
\]
and
\[
 \wt C'_{(\wt\pch, \wt h), 3}  = T^{\wt m + n(\pch) - m}.C'_{(\pch, h), 3},
\]
where $\wt m \sceq m(\wt\pch, \wt h)$ (the contribution from $\wt\shmap$).

If $(\wt\pch', \wt h^{-1})\in\wt\choices$, then, with $\wt m\sceq m(\wt\pch', \wt h^{-1})$, 
\[
 \wt C'_{(\wt\pch', \wt h^{-1}), 1}  = T^{n(\pch')}. C'_{(\pch, h), 2},\quad \wt C'_{(\wt\pch', \wt h^{-1}), 2}  = T^{n(\pch)}. C'_{(\pch, h), 1}
\]
and 
\[
\wt C'_{(\wt\pch', \wt h^{-1}), 3}  = T^{\wt m+n(\pch')}hT^{-m}. C'_{(\pch, h), 3} = T^{\wt m} \wt h T^{n(\pch) - m}. C'_{(\pch, h), 3}.
\]

Suppose now that $\pch$ is a rectangle and let 
\[
 \big( (\pch_j, a_j) \big)_{j=1,\ldots, k}
\]
be the cycle in $\fpch\times\Gamma$ determined by $(\pch, h)$. The inverse cycle, that is the cycle in $\fpch\times\Gamma$ determined by $(\pch, a_k^{-1})$, is 
\[
 \big( (\pch'_j, b_j) \big)_{j=1,\ldots, k}
\]
where (see \cite{Pohl_Symdyn2d})
\[
 \pch'_j = \pch_{k-j+2} \quad\text{and}\quad b_j = a_{k-j+1}^{-1}.
\]
Here, the indices are taken modulo $\cyl(\pch)$. Then there is a unique $j_0\in \{1,\ldots,k\}$ such that $\wt\choices$ contains either $(\wt\pch_{j_0}, \wt a_{j_0})$ or $(\wt\pch'_{j_0}, \wt b_{j_0})$. 

If $(\wt\pch_{j_0}, \wt a_{j_0}) \in \wt\choices$, then 
\[
 \big( (\wt\pch_{j_0+j-1}, \wt a_{j_0+j-1})\big)_{j=1,\ldots, k}
\]
is the cycle in $\wt\fpch\times\Gamma$ determined by $(\wt\pch_{j_0}, \wt a_{j_0})$. In this case we have
\[
 \wt C'_{(\wt\pch_{j_0}, \wt h_{j_0}), j} = T^{n(\pch_{j_0+j-1})}.C'_{(\pch, h), j_0+j-1}
\]
for $j=1,\ldots, k$. 

If $(\wt\pch'_{j_0}, \wt b_{j_0})\in\wt\choices$, then we have
\[
 \wt C'_{(\wt\pch'_{j_0}, \wt b_{j_0}), j} = T^{n(\pch_{k-j_0-j+3})} a_{k-j_0-j+2}.C'_{(\pch, h), k-j_0-j+2}
\]
for $j=1,\ldots, k$. 
\end{remark}

To end, we compare the period functions arising in the presented transfer operator approach to the period functions and functional equations from \cite{Lewis_Zagier, Deitmar_Hilgert, Chang_Mayer_eigen, Mayer_Muehlenbruch_Stroemberg, Bruggeman_Muehlenbruch}. The period functions in \cite{Moeller_Pohl, Pohl_mcf_Gamma0p} are instances of the general result here.

\begin{remark}
For the modular group $\PSL_2(\Z)$, the period functions defined here are identical with those from \cite{Lewis_Zagier} for the choice
\[
 C'\sceq \left\{ X\in S\h\left\vert\ X=a\frac{\partial}{\partial x}\vert_{iy} + b\frac{\partial}{\partial y}\vert_{iy},\ a>0,\ b\in\R,\ y>0\right.\right\},
\]
that is if the set of representatives for the cross section is based on the imaginary axis and points to the right.  Proposition~\ref{isoper} implies that the number of functional equations for the definition of period functions and the number of terms in each functional equation is invariant under all admissible choices. Thus, for $\PSL_2(\Z)$, independent of any choices, there is always a single functional equation and it has three terms. In contrast, the functional equations from \cite{Mayer_Muehlenbruch_Stroemberg, Bruggeman_Muehlenbruch} have four terms. This means that their period functions cannot arise from our constructions. For the same counting reason, the functional equations in \cite{Mayer_Muehlenbruch_Stroemberg} for arbitrary Hecke triangle groups are not special cases of this work. 

For the period functions in \cite{Deitmar_Hilgert, Chang_Mayer_eigen}, a comparison of the structure of the domains of the component functions immediately shows that also these period functions do not arise from the constructions in this article.
\end{remark}

\bibliography{ap_bib}
\bibliographystyle{amsalpha}
\end{document}